\documentclass{amsart}
\usepackage{times}
\usepackage{amssymb,amsmath}

\usepackage{times,bm,mathrsfs}
\usepackage{amsmath}
\usepackage{amssymb}
\usepackage{amsfonts}
\usepackage{mathrsfs}
\usepackage{txfonts}
\usepackage{color}
\usepackage{graphicx}
\usepackage{amscd}
\usepackage{epsfig}

\usepackage{hyperref}

\hypersetup{
colorlinks = true,       
 linkcolor = blue,
   urlcolor=cyan           
   }

\DeclareMathOperator{\Var}{Var}

\newcommand{\spr}{\texttt{sp}}

\newcommand{\oa}{\overline}

\def\path{{\tt path}}

\newcommand{\acal}{\mathcal{A}}

\newcommand{\pcal}{\mathcal{P}}

 \newcommand{\ga}{\alpha}
 \newcommand{\gb}{\beta}
 \newcommand{\gc}{\gamma}

 \newcommand{\gth}{\theta}

 \newcommand{\gl}{\lambda}
 \newcommand{\gL}{\Lambda}

\newcommand{\eps}{\varepsilon}

\newcommand{\bdm}{\begin{displaymath}}
\newcommand{\edm}{\end{displaymath}}
\newcommand{\bea}{\begin{eqnarray*}}
\newcommand{\eea}{\end{eqnarray*}}
\newcommand{\bean}{\begin{eqnarray}}
\newcommand{\eean}{\end{eqnarray}}

\newcommand{\prob}{\mathbb{P}}

\newcommand{\E}{\mathbb{E}}
\newcommand{\I}{\mathbb{I}}
\newcommand{\var}{\mathrm{Var}}

\newcommand{\vr}{\varrho}

\newcommand{\Cov}{\text{Cov}}

\renewcommand{\mod}{\text{ mod }}

\newtheorem{theorem}{Theorem}

\newtheorem{corollary}{Corollary}
\newtheorem{definition}{Definition}

\newtheorem{lemma}{Lemma}

\newtheorem{remark}{Remark}

\begin{document}

\title{Large dimensional random $ k$ circulants}

\author{Arup Bose}
 \address{Stat Math Unit, Kolkata, Indian Statistical Institute.}
 \thanks{AB partially supported by J. C. Bose Fellowship, Government of India.}
\author{Joydip Mitra}
\address{ Management Development Institute
Gurgaon}
 \author{Arnab Sen }
\address{Dept. of Statistics, U.C. Berkeley.}
\date{\today}


\begin{abstract}
Consider random $k$-circulants $A_{k,n}$ with $n \to \infty, k=k(n)$ and whose input sequence
$\{a_l\}_{l \ge 0}$ is independent with mean zero and variance one and $\sup_n n^{-1}\sum_{l=1}^n \E
|a_l|^{2+\delta}< \infty$ for some $\delta > 0$. Under suitable restrictions on the sequence $\{k(n)\}_{ n \ge 1}$, we show
that the limiting spectral distribution (LSD) of the empirical distribution of suitably scaled
eigenvalues  exists and identify the limits.
In particular,  we prove the following: Suppose $g \ge 1$ is fixed and $p_1$ is the smallest prime divisor of $g$.
 Suppose $P_g=\prod_{j=1}^g E_j$ where $\{E_j\}_{1 \le j \le g}$ are i.i.d.\ exponential random variables with mean one.

 (i) If $k^g  = -1+ s n$ where $s=1$ if $g=1$ and $s = o(n^{p_1 -1})$ if $g>1$, then the  empirical spectral distribution of $n^{-1/2}A_{k,n}$ converges weakly in probability to $U_1P_g^{1/2g}$ where  $U_1$
is uniformly distributed over the $(2g)$th roots of unity, independent of $P_g$.

(ii) If  $g \ge 2$ and $k^g  = 1+ s n$ with $s = o(n^{p_1-1})$ then the empirical spectral distribution of $n^{-1/2}A_{k,n}$ converges weakly in probability to $U_2P_g^{1/2g}$   where $U_2$ is
uniformly distributed over the unit circle in $\mathbb R^2$, independent of $P_g$.

On the other hand, if $k \ge 2 $, $ k= n^{o(1)}$ with $\gcd(n,k) = 1$, and the input is i.i.d. standard normal variables, then $F_{n^{-1/2}A_{k,n}}$ converges weakly in probability to the uniform  distribution  over the circle with center at $(0,0)$ and radius $r = \exp( \E [ \log \sqrt E_1] )$.

We
also show that when $n=k^2+1\to \infty$, and the input is i.i.d.\ with finite $(2+\delta)$ moment, then the spectral radius, with appropriate scaling and centering,
converges to the Gumbel distribution. \end{abstract}
\bigskip

\maketitle



\section{Introduction}\label{section:intro}
For any (random)  $n\times n$ matrix $B$, let $\mu_1(B), \ldots , \mu_n(B) \in \mathbb C = \mathbb R^2$ denote its eigenvalues including multiplicities. Then the
empirical spectral distribution (ESD) of $B$ is the (random) distribution function on
$\mathbb R^2$ given by
$$F_{B} (x, y) = n^{-1}  \# \Big \{ j :  \mu_j(B)  \in (-\infty, x] \times (-\infty, y], \ 1 \le j \le n  \Big \}.$$  \noindent For a sequence of random $n \times n$ matrices $\{ B_n\}_{ n \ge 1}$ if
the corresponding ESDs $F_{B_n}$ converge  weakly (either almost surely or in probability) to a (nonrandom) distribution
$F$ in the space of probability measures on $\mathbb R^2$ as $n \rightarrow \infty$, then $F$ is called the limiting spectral distribution (LSD) of $\{B_n\}_{n \ge 1}$.  See Bai
(1999)\cite{Bai99}, Bose and Sen (2007)\cite{Bosesen2008} and Bose, Sen and Gangopadhyay (2009)\cite{bosesengangopadhyay09} for description of several interesting situations where the LSD exists
and can be explicitly specified.

Another important quantity associated with a matrix is its spectral radius.
 For any matrix $B$, its spectral radius $\spr(B)$ is defined as
\[  \spr(B) := \max \Big \{|\mu|: \mu \  \text{ is an eigenvalue of } B \Big\},  \]
where $|z|$ denotes the modulus of
$z \in \mathbb C$.
For classical random matrix models such as the Wigner matrix and i.i.d.\ matrix, the limiting distribution of an appropriately normalized spectral radius is known for the Gaussian entries (see, for example, Forrester(1993)\cite{forrester93}, Johansson (2000)\cite{johansson00}, Tracy and Widom (2000)\cite{tracywidom00} and, Johnstone (2001)\cite{Johnstone01}) which was  later extended by Soshnikov\cite{soshnikov1, soshnikov2} to more general entries.

Suppose $\underline{a} =\{a_l\}_{ l \ge 0}$ is a sequence of real numbers (called the {\it input} sequence). For positive
integers $k$ and $n$,   define the $n \times n$ square matrix
\[A_{k,n}(\underline{a}) = \left[ \begin{array}{cccc}
a_0 & a_1 & \ldots & a_{n-1} \\
a_{n-k} & a_{n-k+1} & \ldots & a_{n-k-1} \\
a_{n-2k} & a_{n-2k+1} &
\ldots & a_{n-2k-1} \\ & & \vdots & \\
\end{array} \right]_{n \times n}. \]
All subscripts appearing in the matrix entries above are calculated modulo $n$.
Our convention will be to start the row and column indices from zero. Thus, the $0$th row of
$A_{k,n}(\underline{a})$ is
$\left( a_0,\ a_1,\ a_2,\ \ldots,\ a_{n-1}
\right).$
For $0 \le j < n-1$, the $(j+1)$-th row of $A_{k,n}$ is
 a right-circular shift of the $j$-th row by $k$ positions (equivalently, $k \mbox{ mod } n$
positions). We will write
$A_{k,n}(\underline{a})=A_{k,n}$ and it is said to be a $k$-\emph{circulant matrix}.
Note that $A_{1, n}$
is the well-known circulant matrix.
Without loss of generality,
$k$ may always be reduced modulo $n$. Our goal is to study the LSD and the distributional limit of the spectral radius of
suitably scaled $k$-circulant matrices $A_{k, n}(\underline{a})$ when the input sequence $\underline{a} = \{a_l\}_{ l\ge 0}$ consists of  i.i.d.\ random variables.
\subsection{Why study $k$-circulants?}
One of the usefulness of circulant matrix stems from
its deep connection to Toeplitz matrix - while the former has an explicit and easy-to-state formula of its spectral decomposition, the spectral analysis of the latter is much harder and challenging in general.
If the input $\{a_l\}_{ l \ge 0}$ is square
summable, then the circulant approximates the corresponding Toeplitz in various senses with the growing dimension. Indeed, this approximating property is exploited to obtain the
LSD of the Toeplitz matrix as the dimension increases. See Gray (2006)\cite{gray06} for a recent and relatively easy account.

When the input sequence is i.i.d. with
positive variance, then it loses the square summability. In that case, while the LSD of the
(symmetric) circulant is normal (see Bose and Mitra (2002)\cite{Bosemitra02} and Massey, Miller
and Sinsheimer (2007)\cite{masseymillersinsheimer07}), the LSD of the (symmetric) Toeplitz is
nonnormal (see Bryc, Dembo and Jiang (2006)\cite{brycdembojiang06} and Hammond and Miller (2005)\cite{hammil05})

On the other hand, consider the random symmetric band Toeplitz matrix, where the banding parameter
$m$, which essentially is a measure of the number of nonzero entries, satisfies $m \to \infty$ and
$m/n\to 0$. Then again, its spectral distribution is approximated well by the corresponding banded
symmetric circulant. See for example Kargin (2009)\cite{kargin09} and Bose and Basak (2009)\cite{bosebasak09}.
Similarly, the LSD of the $(n-1)$-circulant was derived in Bose and Mitra (2002)\cite{Bosemitra02}) (who called it the reverse circulant matrix). This
has been used in the study of symmetric band Hankel matrices. See Bose and Basak (2009)\cite{bosebasak09}.

The circulant matrices  are
diagonalized by the Fourier matrix $F = ((F_{s,t})), F_{s,t} =  e^{2 \pi i st/n}/ \sqrt{n}, 0 \le s, t < n$. Their eigenvalues are the discrete Fourier transform of the input sequence $\{a_l\}_{0 \le l < n}$  and are given by  $\lambda_t= \sum_{l=0}^{n-1} a_l e^{-2\pi it
/n}, 0 \le t < n$. The eigenvalues of the circulant matrices   crop up crucially in time series analysis. For example, the
periodogram of a sequence $\{ a_l\}_{ l \ge 0}$
is defined as $n^{-1}|\sum_{l=0}^{n-1} a_l e^{2\pi ij /n}|^2$, $-\lfloor\frac{n-1}{2}\rfloor \le j \le  \lfloor\frac{n-1}{2}\rfloor$ and
is a simple function of the eigenvalues of the corresponding circulant matrix. The study of the properties of periodogram is fundamental in the spectral analysis of time series. See for instance Fan and Yao (2003)\cite{fanyao03}. The maximum of the perdiogram, in particular, has been studied in Mikosch (1999)\cite{Mikosch99}.

The $k$-circulant matrix and its block versions  arise in many different areas of Mathematics and  Statistics -  from multi-level supersaturated design of experiment (Georgiou and  Koukouvinos (2006) \cite{georgiou06}) to spectra of De Bruijn graphs (Strok (1992)\cite{strok1992circulant}) and $(0,1)$-matrix solutions to $A^m = J_n$ (Wu, Jia  and Li (2002) \cite{wu2002g}) - just to name a few. See also the book   by Davis (1979)\cite{Davis79} and the article by Pollock
(2002)\cite{Pollock02}. The
$k$-circulant matrices
with random input sequence are examples of so called `patterned'
matrices. Deriving LSD for general patterned matrices has drawn
significant attention in the recent literature. See for example the review article by Bai
(1999)\cite{Bai99} or the more recent Bose and Sen (2008)\cite{Bosesen2008} and also Bose, Sen and Gangopadhyay (2009)\cite{bosesengangopadhyay09}).

However, there does not seem to have been any studies of the general random
$k$-circulant either  with respect to the LSD or with respect to the spectral radius. It seems natural to investigate these.
The LSDs of the $1$-circulant and $1$-circulant with symmetry restriction, $(n-1)$ circulant  are known.
It seems interesting to
investigate the possible LSDs that may arise from $k$-circulants. Likewise, the limit distributions
of the spectral radius of circulant and the $(n-1)$-circulant are both Gumbel. It seems natural to
ask what happens to the distributional limit of the spectral radius for general $k$-circulants.
\subsection{Main results and discussion}

\subsubsection{Limiting Spectral distributions}

The LSDs for $k$-circulant matrices  are known for a few important special cases.
If the input sequence $\{a_l\}_{l \ge 0}$ is i.i.d.\ with finite third moment, then the limit distribution of the
circulant matrices ($k=1$) is bivariate normal (Bose and Mitra (2002)\cite{Bosemitra02}). For the
\emph{symmetric} circulant with i.i.d.\ input having finite second moment, the LSD is real normal,
(Bose and Sen (2007)\cite{Bosesen2008}). For the $k$-circulant with $k=n-1$, the LSD is the
symmetric version of the positive square root of
the exponential variable with mean  one  (Bose and Mitra (2002)\cite{Bosemitra02}).

Clearly, for many combinations of  $k$ and
$n$, a lot of eigenvalues are zero. Later  we provide a formula solution for the eigenvalues.
From this, if $k$ is prime and $n=m \times k$ where $\gcd(m,k) =
1$, then $0$ is an eigenvalue with multiplicity $(n-m)$. To avoid this degeneracy and to keep our exposition simple, we primarily restrict our attention to the case when $\gcd(k, n)=1$.

In general, the structure of the eigenvalues depend on the number theoretic relation between $k$ and $n$ and the LSD may vary widely. In particular, LSD is not `continuous' in $k$. In fact, while the ESD of usual circulant matrices $n^{-1/2}A_{1, n}$ is  bivariate normal,  the ESD of  2-circulant matrices $n^{-1/2}A_{2, n}$ for $n$ large odd number looks like a solar ring
(See Figure~\ref{fig:case3&4}). The next theorem tells us that the radial component of the LSD of $k$-circulants with $k \ge 2$ is always degenerate,  at least when the input sequence is i.i.d. normal,  as long as $k = n^{o(1)}$ and $\gcd(k, n)=1$.
\begin{figure}[htp]
\centering
\includegraphics[height=50mm, width =60mm]{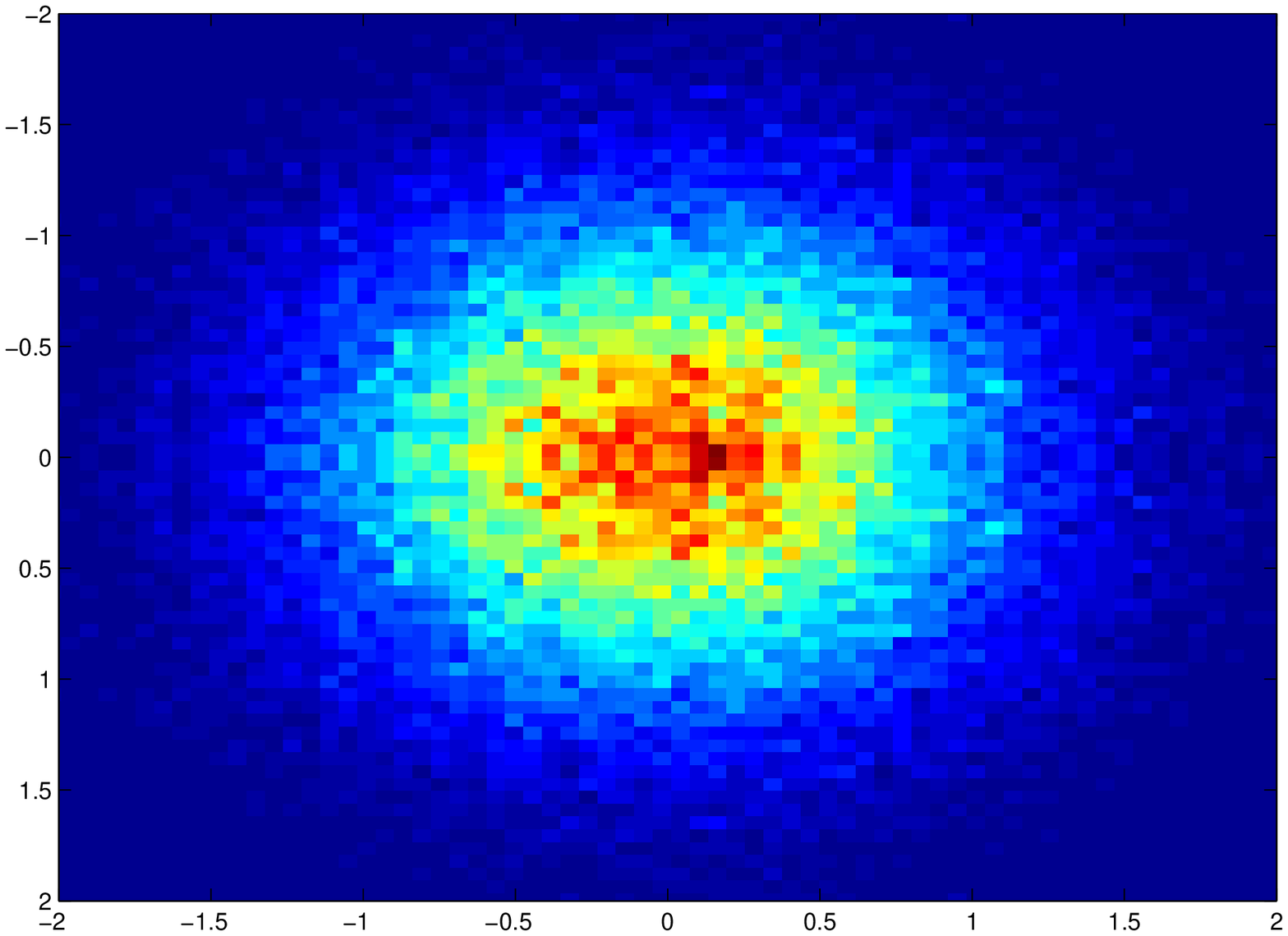}
\hfill
\includegraphics[height=50mm, width =60mm ]{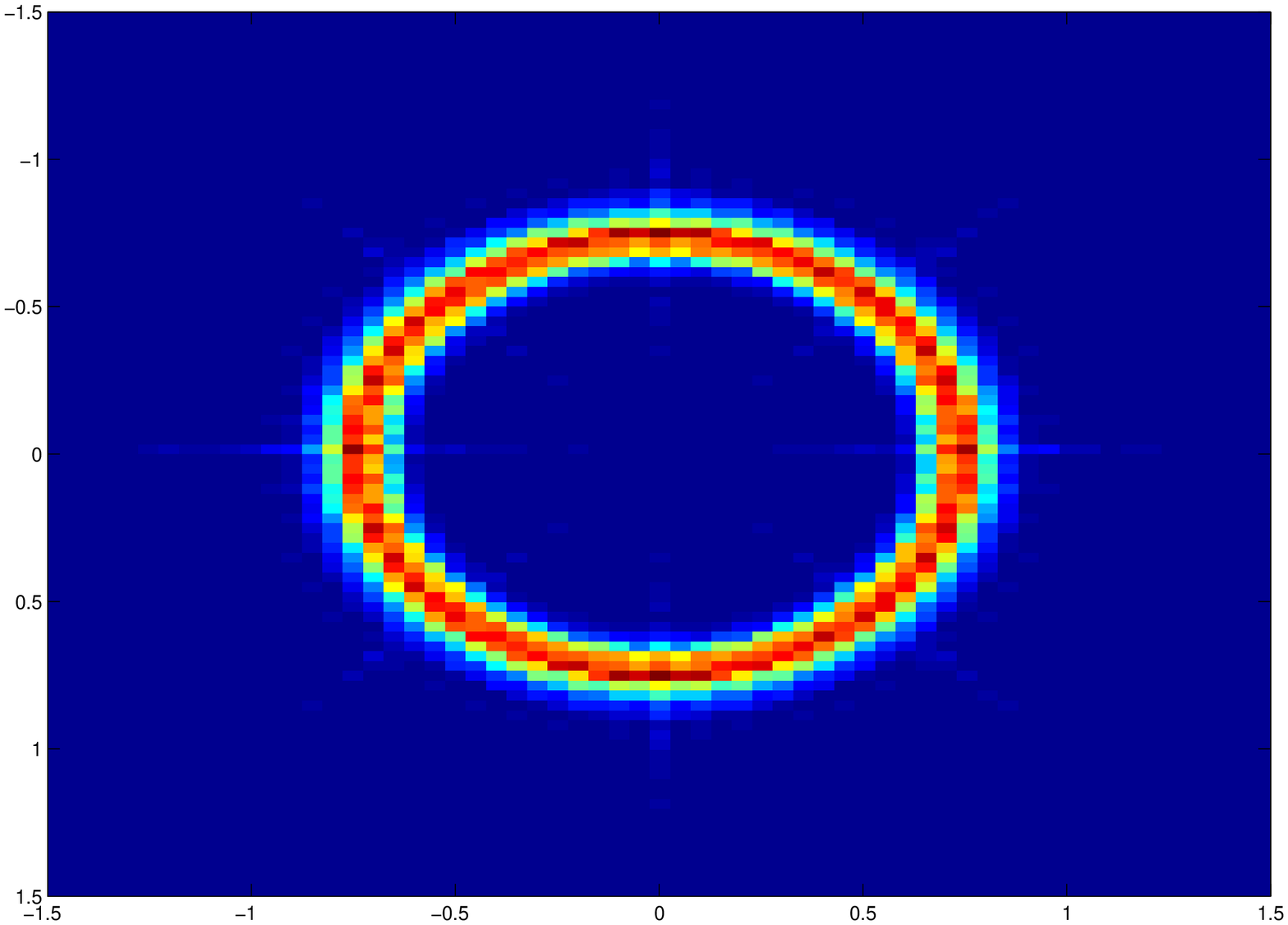}\\

\caption{Eigenvalues of $100$ realizations of $n^{-1/2}A_{k,n}$ with $a_l \sim N(0,1)$ when (i)  $k =1, n=901$ (left) and (ii) $k =2, n=901$ (right). The color represents the height of the histogram - from red (high) to blue (low). } \label{fig:case3&4}
\end{figure}
\begin{theorem}\label{thm:degenerate}
Suppose $\{a_l\}_{ l \ge 0}$ is an i.i.d.\ sequence of  $N(0,1)$ random
variables. Let $k \ge 2 $  be such that $ k= n^{o(1)}$ and $n \to \infty $ with $\gcd(n,k)
= 1$. Then
$F_{n^{-1/2}A_{k,n}}$
converges weakly
in probability
to the
uniform  distribution  over the circle with center at $(0,0)$ and radius $r
= \exp( \E [ \log \sqrt E] )$,
$E$ being an exponential random variable with mean one.
\end{theorem}
\begin{remark}
Since $-\log E$ has the standard Gumbel distribution
which has mean $\gamma $ where $\gamma  \approx 0.57721$ is  the Euler-Mascheroni constant, it follows that $r = e^{ - \gc/2} \approx 0.74930$.
\end{remark}

In view of Theorem~\ref{thm:degenerate}, it is natural to consider the case when $k^{g}= \Omega(n)$ and $\gcd(k, n)=1$ where $g$ is a fixed integer. In the next two theorems, we consider two special cases of the above scenario, namely when $n$ divides $k^g \pm 1$. Consider the following assumption.\\

\noindent  \textbf{Assumption \texttt{I}.}  The sequence $\{a_l\}_{l\ge 0}$ is independent with mean zero, variance
one and for some $ \delta > 0$, $$\sup_n n^{-1}\sum_{i=0}^{n-1} E|a_l|^{2+\delta} <
\infty.$$

We are now ready to state our main theorems on the existence of LSD.

 \begin{theorem} \label{theo:lsd12}
Suppose $\{a_l\}_{l \ge 0}$
satisfies Assumption \texttt{I}. Fix $g \ge 1$ and let $p_1$ be the smallest prime divisor of $g$.
 Suppose  $k^g  = -1+ s n$ where $s=1$ if $g=1$ and $s = o(n^{p_1 -1})$ if $g>1$.
 Then $F_{n^{-1/2}A_{k,n}}$ converges weakly in probability to
 $U_1(\prod_{j=1}^g E_j)^{1/2g}$  as $n \to \infty$ where $\{ E_j\}_{ 1 \le j \le g}$ are i.i.d.\ exponentials with mean one and $U_1$
is uniformly distributed over the $(2g)$th roots of unity, independent of $\{E_j\}_{ 1 \le j \le g}$.
\end{theorem}

 \begin{theorem}\label{theo:lsd45}
Suppose $\{a_l\}_{l \ge 0}$ satisfies Assumption \texttt{I}. Fix $g \ge 1$ and let $p_1$ be  the smallest prime divisor of $g$. Suppose  $k^g  = 1+ s n$ where  $s =0$ if $g=1$ and $s = o(n^{p_1-1})$ if $g>1$. Then $F_{n^{-1/2}A_{k,n}}$ converges weakly in probability to
$U_2(\prod_{j=1}^g E_j)^{1/2g}$  as $n \to \infty$ where $\{E_j\}_{1 \le j \le g}$ are i.i.d.\ exponentials with mean one and $U_2$ is
uniformly distributed over the unit circle in $\mathbb R^2$, independent of $\{E_j\}_{1 \le j \le g}$.
\end{theorem}

\begin{figure}[htp]
\centering
\includegraphics[height=50mm, width =60mm ]{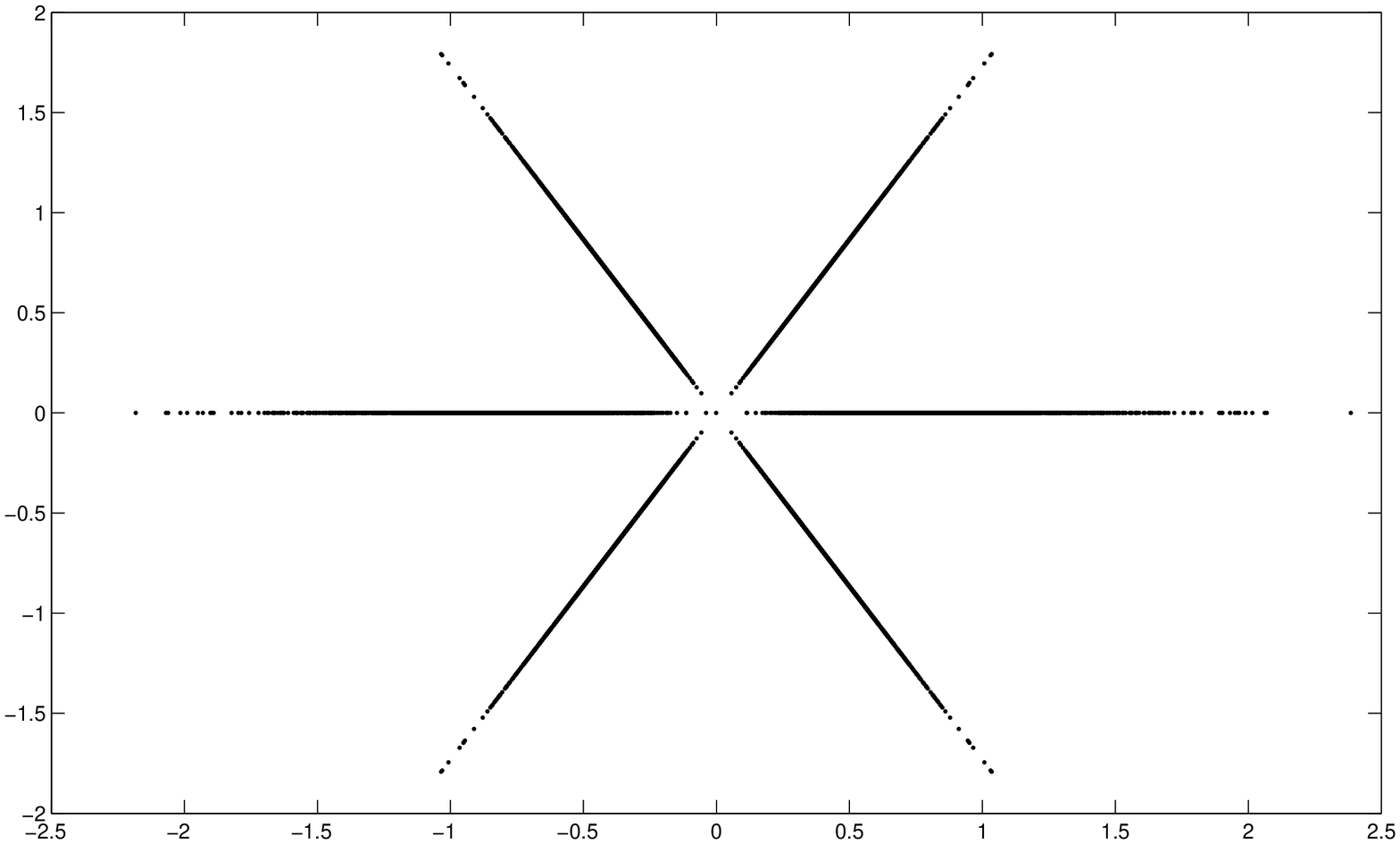}
\hfill
\includegraphics[height=50mm, width =60mm ]{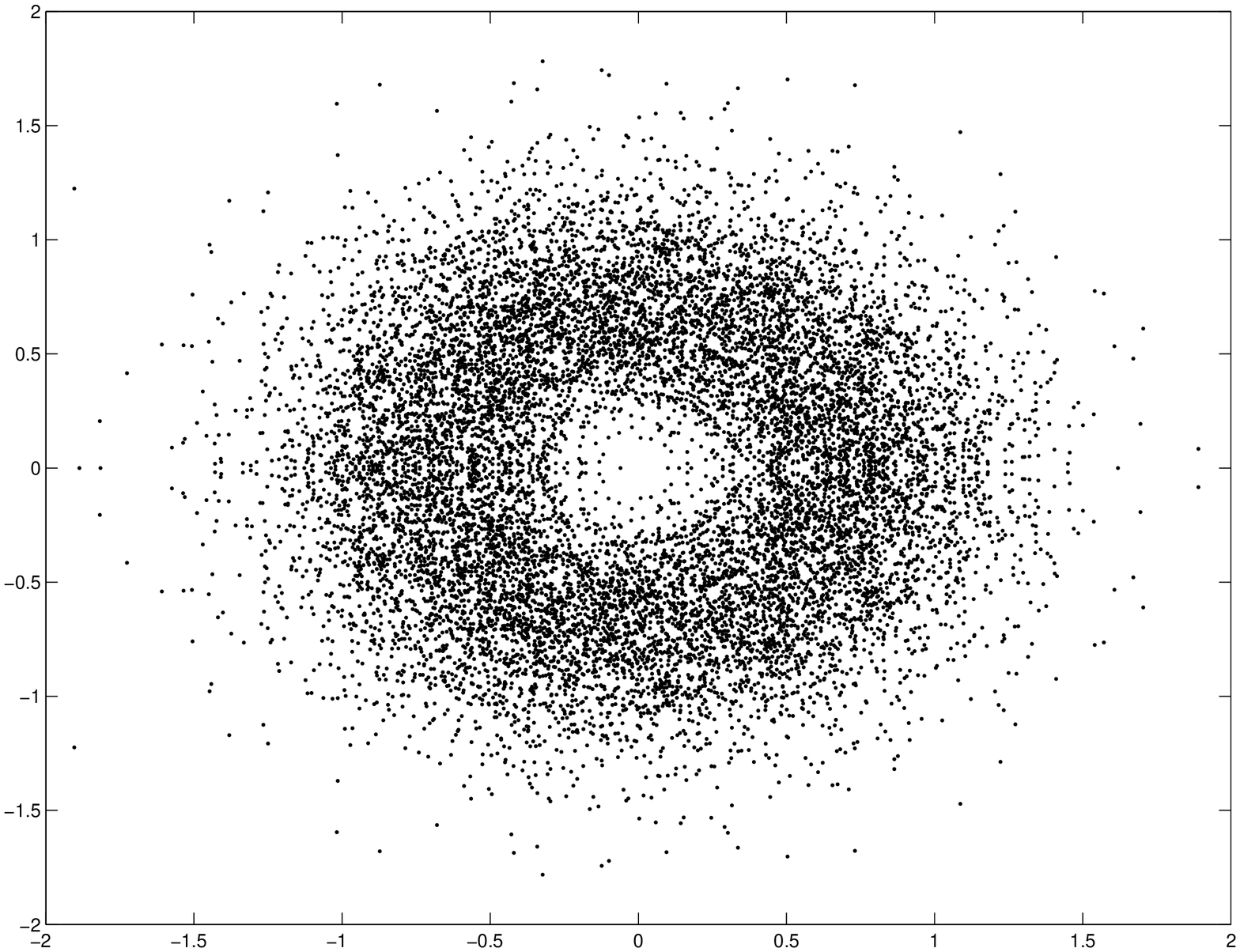}\\

\caption{Eigenvalues of $20$ realizations of $n^{-1/2}A_{k,n}$ with $a_l \sim \text{Exp}(1) -1$ when (i)  $k =11, k^3 = -1 +2n$ (left) and (ii) $k =11, k^3 = 1 +2n$ (right).} \label{fig:case1&2}
\end{figure}

\begin{remark}
(1) Theorem~\ref{theo:lsd12} and Theorem~\ref{theo:lsd45} recover the LSDs of $k$-circulants for $k=n-1$ and $k=1$ respectively.

\noindent (2)
While the radial coordinates of the LSD  described in Theorem \ref{theo:lsd12}  and
\ref{theo:lsd45} are same, their angular coordinates differ.
While
one puts its mass only at
discrete places $e^{i 2 \pi j/ 2g}, 1 \le j \le 2g$ on the unit circle, the other spreads  its mass
uniformly over the entire unit circle. See Figure \ref{fig:case1&2}.

\noindent (3) The restriction on $s = (k^g \pm 1)/n$ in the above two theorems seems to be a natural one. Suppose $g$ is a prime and so $g = p_1$. In this case if $s \ge n^{p_1 -1} $, then $k$ becomes greater than or equal to $n$ violating the assumption that $k < n$.

\noindent (4)  We cannot expect similar LSDs to hold for more general cases like $k^g = \pm r + n, r >1$ fixed. Compare  Figure \ref{fig:case1&2} and Figure \ref{fig:case5&6}.
\end{remark}
\begin{figure}
\centering
\includegraphics[height=50mm, width =60mm]{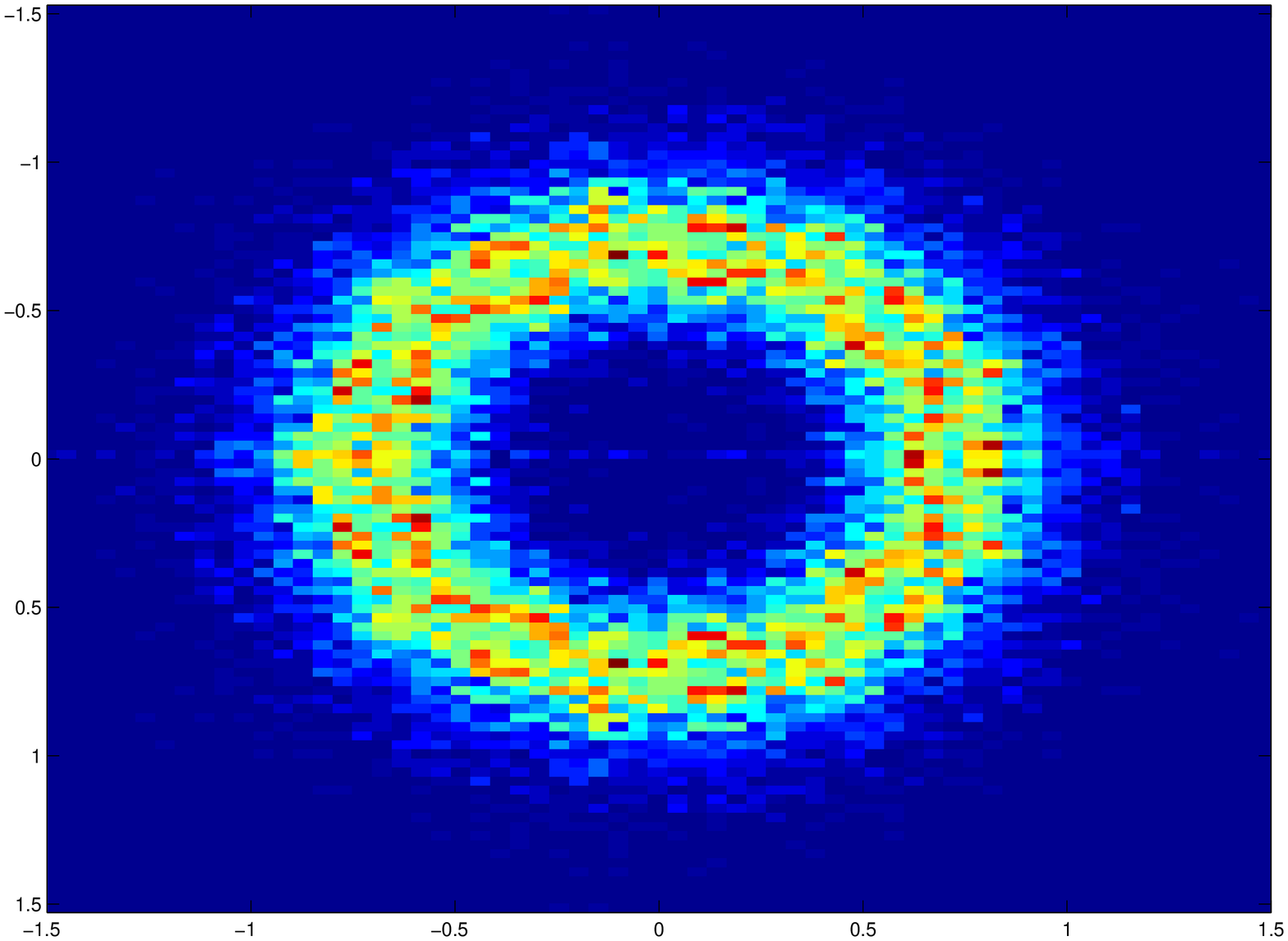}
\hfill
\includegraphics[height=50mm, width =60mm ]{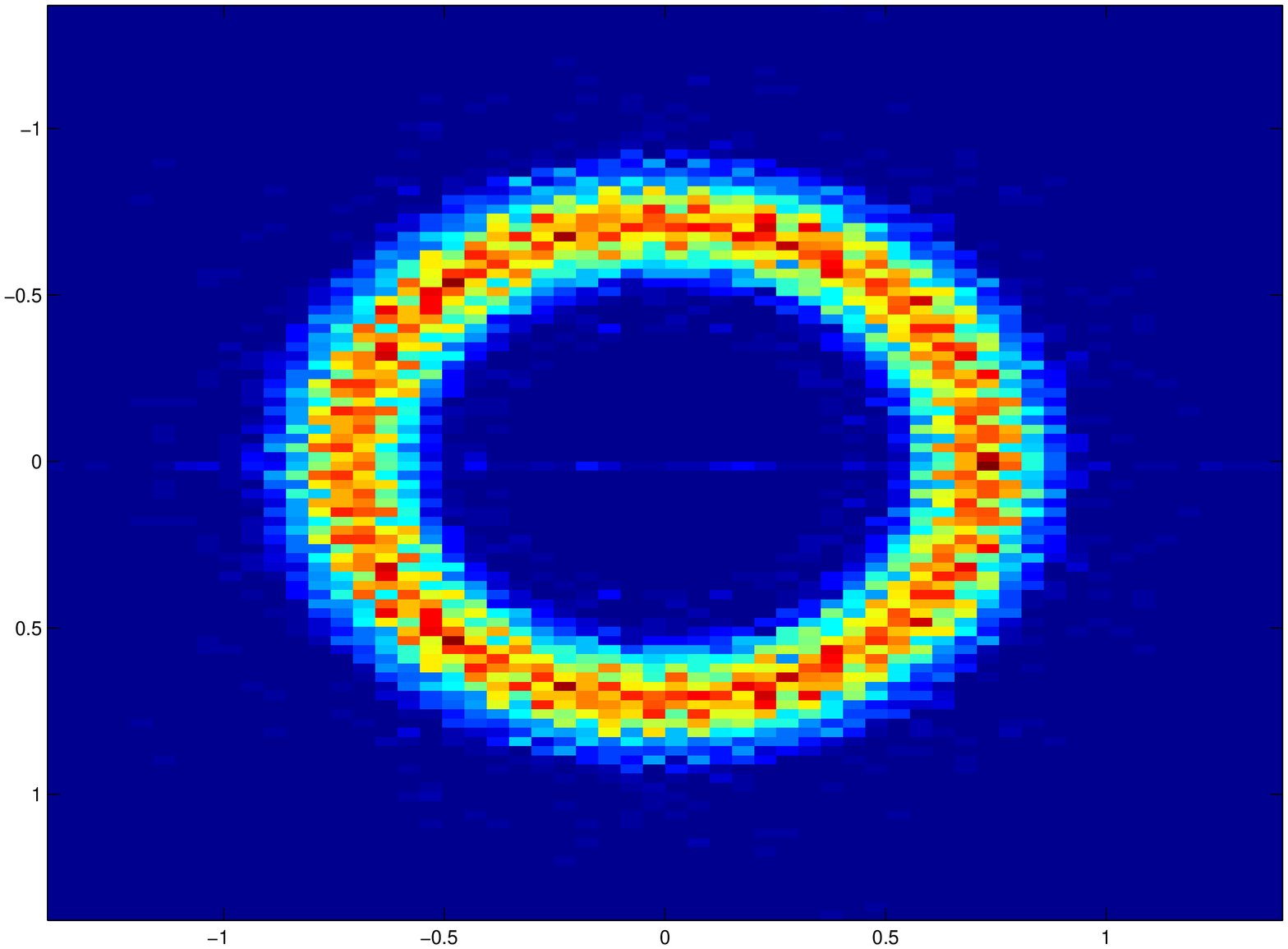}\\
\caption{Eigenvalues of $100$ realizations of $n^{-1/2}A_{k,n}$ with $a_l \sim N(0,1)$ when (i)  $k =16, n=-3 +k^2$ (left) and (ii) $k =16, n=3+k^2$ (right). The color represents the height of the histogram - from red (high) to blue (low). } \label{fig:case5&6}
\end{figure}
\subsubsection{Spectral radius}
For the $k$-circulant, first suppose that the input sequence is i.i.d. standard normal.
When $k=1$, it is easy to check that the modulus square of the eigenvalues are
exponentials and they are independent of each other. Hence,
the appropriately scaled and normalized spectral radius converges to the Gumbel distribution. But when
the input sequence is i.i.d.\
but not necessary normal, that  independence structure is lost. A  careful use of  Koml\'{o}s-Major-Tus\'{a}ndi type sharp normal approximation results are needed to deal with this case. See Davis and Mikosch (1999)\cite{Mikosch99}.
These approximations imply
that the limit
continues to be Gumbel.
The spectral radius of the
$(n-1)$-circulant is the same as that of the circulant and hence it has the same limit.
See also Bryc and Sethuraman (2009)\cite{brycsethuraman09}
who use the same approach
for the symmetric circulant.

Now let $g=2$ and for further simplicity, assume that $n = k^2+1$.
If the input sequence is i.i.d. standard normal, then
the modulus of the nonzero eigenvalues are
independent and distributed according
to $(E_1E_2)^{1/4}$, where $E_j, j=1,2$  are i.i.d.\ standard exponential.
Thus,
the behavior of the spectral radius is the same as that of the
maxima of i.i.d variables each distributed as
$(E_1E_2)^{1/4}$. This
is governed by the tail behaviour of $E_1E_2$.
We deduce this tail behaviour via properties of Bessel functions and the limit again turns out to be Gumbel. Now, as
suggested by the results of Davis and Mikosch (1999)\cite{Mikosch99}, even when the input sequence
is only assumed to be i.i.d. and not necessarily normal, with suitable moment condition, some kind
of invariance principle holds and the same limit persists. We show that this is indeed the case.
\begin{theorem} \label{theo:max} Suppose $\{a_l\}_{l\ge 0}$ is an
i.i.d.\ sequence of random variables with mean zero and variance $1$ and $ \E |a_l|^\gc < \infty $
for some $\gc > 2$. If $n=k^2+1$ then
\[ \frac{ \spr( n^{-1/2}A_{k, n}) -d_q } { c_q}  \]
converges in distribution to the standard Gumbel  as $n \to \infty$ where $ q =q(n) = { \lfloor \frac{n}{4} \rfloor}$ and the normalizing constants $ c_n $ and $d_n$ can be
taken as follows
\begin{equation}\label{eq:normalize}
c_n =(8\log n)^{-1/2}\ \ \text{and} \ \ d_n=\frac{ (\log n)^{1/2} }{\sqrt 2}\left
(1+\frac{1}{4}\frac{\log\log n}{\log n}\right)+\frac{1}{2(8\log n)^{1/2}} \log \frac{\pi}{2}.
\end{equation}
\end{theorem}
In the next section
 we state the basic eigenvalue formula for $k$-circulant
and develop some essential
properties of the eigenvalues. In Section \ref{sec:lsd} and Section \ref{sec:spectralnorm} we state and prove the results on LSD  and the spectral radius respectively. An Appendix reproves the known eigenvalue
formula for $k$-circulant.

\section{Eigenvalues of the $k$-circulant}\label{section:eigenvalues}
We first describe the eigenvalues of a $k$-circulant and prove some related auxiliary properties.
The formula solution, in particular is already known, see for example Zhou (1996)\cite{Zhou96}. We provide a
more detailed analysis which we later use in our study of the LSD and the spectral radius.
Let
\begin{equation}\label{eq:lambda}
\omega=\omega_n := \cos(2\pi/n)+i\sin(2\pi/n),\ i^2=-1\ \  \text{and} \ \
\lambda_t = \sum\limits_{l=0}^{n-1} a_{l} \omega^{tl}, \ \
0 \le t < n.
\end{equation}
\begin{remark}
Note that $\{ \lambda_t, 0 \le t < n \}$ are eigenvalues of the usual
circulant matrix $A_{1,n}$.
\end{remark}
Let $p_1 <p_2<\cdots < p_c$ be all the
common prime factors of $n$ and $k$. Then we may write,
\begin{equation}\label{eq:decomposition} n=n^{\prime} \prod_{q=1}^{c} p_q^{\beta_q} \ \text{  and  } \ \  k=k^{\prime}
\prod_{q=1}^{c} p_q^{\alpha_q}.\end{equation}
Here $\alpha_q,\ \beta_q \ge 1$ and $n^{\prime}$, $k^{\prime}$,
$p_q$ are pairwise relatively prime. We will show that $(n-n^{\prime})$ eigenvalues of $A_{k,n}$ are
zero and $n^{\prime}$ eigenvalues are non-zero functions of $\underline{a}$.

To identify
the non-zero eigenvalues of $A_{k,n}$, we need some preparation.
For any positive integer $m$, the set $\mathbb Z_{m}$ has its usual meaning, that is, $\mathbb Z_{m}  = \{ 0, 1,2, \ldots, m-1\}.$ We introduce the following family of sets
\begin{equation}\label{eq:Sx} S(x) :=
\big  \{xk^b \mod n^{\prime}: b \ge 0 \big  \}, \quad x \in \mathbb Z_{n'}.
\end{equation} We
observe the following facts about the family of  sets $\{ S(x) \}_{ x \in \mathbb Z_{n'}}$.\\

\noindent (I) Let $g_x = \#S(x)$. We call $g_x$  the {\it order} of $x$. Note that $g_0 = 1$. It is
easy to see that
$$S(x) = \{xk^b \mod  n^{\prime}: 0 \le b < g_x \}.$$
An alternative description of $g_x$, which we will use later extensively,
 is the following. For $x \in \mathbb Z_{n^{\prime}}$, let
\[   \mathcal O_x = \{b > 0 \ : b \text{ is an integer and }   xk^b = x \mod  n'  \}.
\]
Then $g_x = \min O_x$, that is,  $g_x$ is the smallest positive integer $b$ such
that $xk^b = x \mod n'$.\\

\noindent (II)
The distinct sets from the collection $\{S(x) \}_{ x \in \mathbb Z_{n'}}$ forms a partition of $\mathbb
Z_{n^{\prime}}$.  To see this, first note that $x \in S(x)$ and hence $\bigcup_{x \in \mathbb Z_{n'}}S(x) =  \mathbb Z_{n'}$. Now suppose $S(x) \cap S(y) \neq \emptyset$. Then, $xk^{b_1} = yk^{b_2}
\mod  n^{\prime}$ for some integers $b_1, b_2 \ge 1$. Multiplying both sides by $k^{g_x-b_1}$
we see that, $x \in S(y)$ so that, $S(x) \subseteq S(y)$. Hence, reversing the roles, $S(x)=S(y)$.

We call the distinct sets in $\{S(x) \}_{ x \in \mathbb Z_{n'}}$ the {\em eigenvalue partition} of $\mathbb Z_{n^{\prime}}$ and denote the partitioning sets and their sizes by
\begin{equation}\label{eq:partitionsets}\pcal_0 = \{0\}, \pcal_1, \ldots, \pcal_{\ell -1}\ \ \text{and}\ \ n_j = \#\pcal_j, \  0 \le j < \ell.\end{equation}
Define \begin{equation}\label{eq:xj} \Pi_j := \prod_{t \in \pcal_j} \lambda_{tn/n'}, \ \ j=0, 1, \ldots , \ell-1.\end{equation} The following theorem provides the formula solution for the
eigenvalues of $A_{k,n}$.
Since this is from a Chinese article
which may not be easily accessible to all readers,
we have provided a proof in the
Appendix.
\begin{theorem}[Zhou (1996)\cite{Zhou96}] \label{theo:formula}
 The characteristic polynomial of $A_{k,n}$ is given by
 \begin{equation}\label{eq:evalue}
 \chiup \left(A_{k,n}
\right) (\lambda)=\lambda^{n-n^{\prime}} \prod_{j=0}^{\ell-1} \left( \lambda^{n_j} - \Pi_j \right).
\end{equation}
\end{theorem}
\subsection{Some properties of the eigenvalue partition $\{\pcal_j, 0 \le j < \ell\}$}
We collect some simple but useful properties about the eigenvalue partition in the following lemma.
\begin{lemma}\label{lem:partitionproperties}
 (i) Let $x, y \in \mathbb Z_{n'}$. If  $n'-t_0 \in S(y)$ \ for some $t_0 \in S(x)$, then for every $t \in S(x) $, we have $n' - t \in S(y)$. \\

\noindent (ii)  Fix  $x \in  \mathbb Z_{n'} $. Then $g_x$ divides $g$ for every $g \in \mathcal O_x$. Furthermore, $g_1$ divides $g_x$ for each $x \in \mathbb Z_{n'} $. \\

\noindent (iii)  Suppose $g$ divides $g_1$. Set $m := \gcd(k^g-1,n')$. Let $X(g)$ and $Y(g)$ be
defined as \begin{eqnarray}
X(g):= \Big \{x: x \in\mathbb Z_{n'}   \ \ \text{and}\ \  x \ \ \text{has  order} \ \ g \Big \}, \ \
Y(g) := \Big \{ bn'/m \ : \ 0 \le b < m  \Big \}.
\end{eqnarray}
Then $$X(g) \subseteq Y(g), \ \  \#Y(g) = m\ \ \text{and}\ \ \bigcup_{h : h|g} X(h) = Y(g).$$
\end{lemma}
\begin{proof} (i) Since $t \in S(x)  = S(t_0)$, we can write $t = t_0k^b \mod n' $  for some $b \ge 0$. Therefore,   $n'-t = (n'-t_0) k^b \mod n'   \in S(n' - t_0) = S(y)$.

\noindent (ii) Fix $g \in \mathcal O_x$. Since $g_x$ is the smallest element of $\mathcal O_x$, it follows that $g_x \le g$. Suppose, if possible,  $g = qg_x+r$ where $0 < r < g_x$. By the fact $x g_x  = x \mod n'$, it then follows that
\[
x  =  xk^{g}  \mod  n^{\prime} =  xk^{qg_x +r }   \mod  n^{\prime}
   =  xk^r  \mod  n^{\prime}. \]
This implies that  $r \in \mathcal O_x$ and $r <g_x$ which is a contradiction to the fact that $g_x$ is
the smallest element in $\mathcal O_x$. Hence, we must have $r=0$ proving that $g$ divides $g_1$.

Note that $k^{g_1} = 1 \mod  n'$, implying that $xk^{g_1} = x \mod n'$. Therefore $g_1 \in \mathcal O_x$ proving the assertion.

\noindent (iii)
Clearly, $\#Y(g) = m$. Fix $x \in X(h)$  where $h$ divides $g$. Then, $xk^g = x(k^h)^{g/h} = x \mbox{ mod }
n^{\prime}$, since $g/h$ is a positive integer. Therefore $n'$ divides $x(k^g-1)$. So, $n'/m$ divides $x(k^g-1)/m$. But $n'/m$ is
relatively prime to $(k^g-1)/m$ and hence $n'/m$ divides $x$. So, $x=bn'/m$ for some integer $b \ge
0$. Since $0 \le x <n'$, we have $0 \le b < m$, and $x \in Y(g)$, proving $\bigcup_{h : h|g} X(h) \subseteq Y(g)$ and in particular, $ X(g) \subseteq Y(g)$.

On the other hand, take $ 0 \le b < g$. Then $\left( bn'/m\right ) k^g  = \left( bn'/m\right)\mod
n'$. Hence, $g \in \mathcal O_{bn'/m}$ which implies, by part (ii) of the lemma, that $ g_{cn'/m}$ divides $g$. Therefore, $ Y(g) \subseteq \bigcup_{h :
h|g} X(h) $ which completes the proof. \hfill $\Box$
\end{proof}
\begin{lemma} \label{theo:count}
Let $g_1 = q_1^{\gamma_1}  q_2^{\gamma_2}\ldots q_m^{\gamma_m}$
where $q_1 < q_2 < \ldots < q_m$ are primes. Define for $1 \le j \le m$, \[ L_j := \left\{ q_{i_1}q_{i_2} \cdots q_{i_j} : 1 \le i_1 < \ldots < i_j \le m \right\} \] and
 \[G_j = \sum\limits_{l_j \in L_j} \#
Y(g_1/\ell_j) = \sum\limits_{l_j \in L_j} \gcd \left(k^{g_1/\ell_j}-1,n' \right). \]
 Then we have

\noindent
(i) $\# \left\{ x \in \mathbb Z_{n^{\prime}}: g_x < g_1 \right\} = G_1 -
G_2 + G_3 - G_4 + \cdots$.\\

\noindent
(ii) $G_1 - G_2 + G_3 - G_4 + \cdots \le
G_1.$
\end{lemma}
\begin{proof} Fix $x \in \mathbb Z_{n'} $. By Lemma~\ref{lem:partitionproperties}(ii), $g_x$ divides $g_1$ and hence we can write $g_x = q_1^{\eta_1} \ldots
q_m^{\eta_m}$ where, $0 \le \eta_b \le \gamma_b$ for $1 \le b \le
m$. Since $g_x < g_1$, there is at least one $b$ so that $\eta_b<
\gamma_b$. Suppose that exactly $h$-many $\eta$'s are equal to the
corresponding $\gamma$'s where $0\le h < m$. To keep notation simple, we will assume
that, $\eta_b=\gamma_b,\ 1 \le b \le h$ and $\eta_b<\gamma_b,\ h+1
\le b \le m$.

\noindent (i) Then $x \in Y(g_1/q_b)$ for $h+1 \le b \le m$ and $x
\not \in Y(g_1/q_b)$ for $1 \le b \le h$. So, $x$ is counted
$(m-h)$ times in $G_1$. Similarly, $x$ is counted ${m-h \choose 2}$ times in $G_2$, ${m-h \choose 3}$ times in $G_3$, and so on. Hence, total
number of times $x$ is counted in $(G_1-G_2+G_3-\ldots)$ is \[
{m-h \choose 1} - {m-h \choose 2} + {m-h \choose 3}- \ldots = 1. \]
(ii) Note that $m-h \ge 1$. Further, each element in the set
$\left\{ x \in \mathbb Z_{n^{\prime}}: g_x < g_1 \right\}$ is
counted once in $G_1-G_2+G_3- \ldots$ and $(m-h)$  times
in $G_1$. The result follows immediately.
\end{proof}

 \subsection{Asymptotic negligibility of lower order elements}
 We will now consider the elements in $\mathbb Z_{n{\prime}}$ with
order less than that of $1 \in \mathbb Z_{n{\prime}} $ which has the highest order
$g_1$. We will need the proportion of such elements in $\mathbb
Z_{n{\prime}}$. So, we define
\begin{equation}\label{eq:upsilon}
\upsilon_{k,n'} :=\frac{1}{n'}\# \{ x \in\mathbb Z_{n{\prime}}  : g_x < g_1 \}.
\end{equation}
To derive the LSD in the special cases we have in mind, the asymptotic negligibility of
$\upsilon_{k,n'}$ turns out to be important. The following two lemmas establish upper bounds on
$\upsilon_{k,n'}$ and will be crucially used  later.
\begin{lemma} \label{lem:minus_one}
(i) If $g_1 =2$, then $\displaystyle \upsilon_{k,n'} = \gcd(k-1, n')/n'$.\\

\noindent (ii) If $g_1 \ge 4$ is even, and $k^{g_1/2} = -1 \mbox{ mod } n$, then $\displaystyle
\upsilon_{k,n'} \le 1 + \sum_{b| g_1, \ b \ge 3} \gcd(k^{g_1/b} -1, n').$

\noindent (iii) If $g_1 \ge 2$ and $q_1$ is the smallest prime
divisor of $g_1$, then $\displaystyle \upsilon_{k,n'} <2{n'}^{-1}k^{g_1/q_1}.$
\end{lemma}
\begin{proof}  Part (i) is immediate from Lemma
\ref{theo:count} which asserts that $n' \upsilon_{k,n'} = \# Y(1)  = \gcd (k -1, n')$.

\noindent (ii) Fix $x \in \mathbb Z_{n'}$ with $g_x < g_1$. Since $g_x$ divides $g_1$ and $g_x <g_1$, $g_x$ must be of the
form $g_1/b$ for some integer $b \ge 2$ provided  $g_1/b$ is an integer. If $b=2$, then
$ xk^{g_1/2} = x k^{g_x}  =  x  \mod  n^{\prime}$. But $k^{g_1/2} = -1
\mbox{ mod } n^{\prime}$ and so, $xk^{g_1/2} = -x \mbox{ mod }
n^{\prime}$. Therefore, $2x = 0 \mbox{ mod } n^{\prime}$ and $x$
can be either $0$ or $n^{\prime}/2$, provided, of course,
$n^{\prime}/2$ is an integer. But $ g_0 =1 < 2 \le g_1/2$ so $x$ cannot be $0$. So, there is at most one element in the set $X(g_1/2)$. Thus we  have,
\begin{align*} \# \{
x \in \mathbb Z_{n^{\prime}} : g_x < g_1 \}  & = \#X(g_1/2)
+ \sum_{b| g_1, \ b \ge 3} \# \{ x \in  \mathbb Z_{n^{\prime}} : g_x
=g_1/b \}  \\
& = \#X(g_1/2) + \sum_{b| g_1, \ b \ge 3}
\# X(g_1/b)   \\
&
\le  1+ \sum_{b| g_1, \ b \ge 3} \#Y(g_1/b)  \quad [ \text{by Lemma \ref{lem:partitionproperties}(iii)}] \\
&=  1 + \sum_{b| g_1, \ b \ge 3} \gcd(k^{g_1/b} -1, n') \quad [ \text{by Lemma \ref{lem:partitionproperties}(iii)}. ]
\end{align*}
(iii) As in Lemma \ref{theo:count}, let $g_1 = q_1^{\gamma_1}
q_2^{\gamma_2}\ldots q_m^{\gamma_m}$ where $q_1 < q_2 < \ldots <
q_m$ are primes. Then by Lemma \ref{theo:count},
\begin{align*} n^{\prime} \times \upsilon_{k,n^{\prime}}  =  G_1-G_2+G_3-G_4+ \ldots \le G_1 & = \sum\limits_{b=1}^{m} \gcd(k^{g_1/q_b}-1,n^{\prime})
\\ & <  \sum\limits_{b=1}^{m} k^{g_1/q_b}   \le  2k^{g_1/q_1}
\end{align*}
where the last inequality follows from the observation
\[  \sum\limits_{b=1}^{m} k^{g_1/q_b} \le k^{g_1/q_1}\sum\limits_{b=1}^{m} k^{ - g_1(q_b - q_1)/q_1 q_b}  \le k^{g_1/q_1}\sum\limits_{b=1}^{m} k^{ - (q_b - q_1) } \le k^{g_1/q_1} \sum\limits_{b=1}^{m} k^{ - (b - 1) } \le 2 k^{g_1/q_1}. \]
\end{proof}
\begin{lemma}\label{lem:gcdab}
Let  $b$ and $c$ be two fixed positive integers. Then for any integer  $k \ge2 $, the following inequality holds  in each of the four cases,
\[\gcd(k^b \pm 1, k^c \pm 1) \le k^{ \gcd( b, c)} +1. \]
\end{lemma}
\begin{proof}
The assertion trivially follows if one of $b$ and $c$ divides other. So, we assume, without loss, that $b < c$ and $b$ does not divide $c$. Since, $k^c \pm1  =k^{c-b}(k^b +1) + (-k^{c-b}  \pm 1)$,  we can write
\[  \gcd(k^b + 1, k^c \pm 1)   = \gcd(k^b + 1, k^{c - b} \mp 1).  \]
Similarly, \[  \gcd(k^b - 1, k^c \pm 1)   = \gcd(k^b - 1, k^{c - b} \pm 1).  \]
Moreover, if we write  $c_1 = c - \lfloor c/b \rfloor b$, then  by repeating the above step  $\lfloor c/b \rfloor$ times, we can see that  $\gcd(k^b \pm 1, k^c \pm 1) $ is equal to one of $\gcd(k^b \pm 1, k^{c_1} \pm 1)$. Now if $c_1$ divides $b$, then $\gcd(b, c) = c_1$ and we are done. Otherwise, we can now repeat the whole argument  with $b = c_1$ and $c = b$ to deduce that
$\gcd(k^b \pm 1, k^{c_1} \pm 1)$ is one of $\gcd(k^{b_1} \pm 1, k^{c_1} \pm 1)$ where $b_1 =  b - \lfloor b/c_1 \rfloor c_1$. We continue in the similar fashion by reducing each time one of the two exponents of $k$ in the gcd and  the lemma  follows once we recall Euclid's recursive algorithm for computing the  gcd of two numbers.
\end{proof}
\begin{lemma} \label{lem:fkn}
(i) Fix $g \ge 1$. Suppose $k^g  = -1+ s n$,  $n \to \infty$ with $s=1$ if $g=1$ and $s = o(n^{p_1 -1})$ if $g>1$ where $p_1$ is the smallest prime divisor of $g$. Then $g_1=2g$ for all but finitely many $n$ and  $\upsilon_{k,n} \to 0.$\\

\noindent
(ii) Suppose $k^g  = 1+ s n$, $g \geq 1$ fixed, $n \to \infty$ with  $s=0$ if $g=1$ and $s = o(n^{p_1 -1})$ where $p_1$ is the smallest prime divisor of $g$. Then $g_1=g$  for all but finitely many $n$ and $\upsilon_{k,n} \to 0.$
\end{lemma}
\begin{proof} (i) First note that $\gcd(n, k)=1$ and therefore
$n'=n$.  When $g=1$, it is easy to check that $g_1$ =2 and by Lemma~\ref{lem:minus_one}(i), $\upsilon_{k, n} \le 2/n$.

Now assume $g>1$. Since $k^{2g} = (sn-1)^2 =1 \mod n$, $g_1$ divides $2g$.
Observe that $g_1 \ne g =2g/2$ because $k^g = -1 \mod n$.

If $ g_1 = 2g /b$,  where $b$ divides $g$ and $b \ge 3$, then by Lemma~\ref{lem:gcdab},
\[ \gcd(k^{g_1}-1, n) = \gcd\big (k^{2g/b}-1, (k^{g}+1)/s \big) \le \gcd\big (k^{2g/b}-1, k^{g}+1\big ) \le k^{ \gcd(2g/b, \ g)} +1.\]
Note that since  $\gcd(2g/b, g)$ divides $g$ and $\gcd(2g/b, g) \le 2g/b < g$, we have $\gcd(2g/b, g) \le g/p_1$. Consequently,
\begin{equation}\label{eq:2gb_neg}
\gcd(k^{2g/b}-1, n)  \le k^{g/p_1} +1 \le (sn - 1)^{1/p_1} +1  = o(n),
\end{equation}
which is a contradiction to the fact that $k^{g_1} = 1 \mod n $ which implies that
$\gcd(k^{g_1}-1, n) = n$. Hence, $g_1 = 2g$.
Now by Lemma \ref{lem:minus_one}(ii) it is enough to show that for any fixed $ b \ge 3$ so that $b$ divides $g_1$,
 \[ \gcd(k^{g_1/b} -1, n)/n  = o(1) \ \ \text{as}\ \  n \to \infty,\]
which we have already proved in \eqref{eq:2gb_neg}.

\noindent
(ii)  Again $\gcd(n, k)=1$ and
$n'=n$.  The case when $g=1$ is trivial as then we have $g_x = 1$ for all $x \in \mathbb Z_n$ and $\upsilon_{k,n}= 0$.

Since $k^g = 1 \mod n$, $ g_1$ divides  $g$.  If $g_1< g$, then $g_1 \le g/p_1$ which implies that $k^{g_1} \le k^{g/p_1} = (sn+1)^{1/p_1} = o(n)$, which is a contradiction. Thus, $g = g_1$.

  Now Lemma \ref{lem:minus_one}(iii)
immediately yields,
 \[ \upsilon_{k,n} < \frac{2 k^{g_1/p_1}}{n} \le \frac{2 (1+
sn)^{1/p_1}}{n}  = o(1). \]
\end{proof}

\section{Proof of Theorem~\ref{thm:degenerate}, Theorem~\ref{theo:lsd12} and Theorem~\ref{theo:lsd45}}\label{sec:lsd}

\subsection{Properties of eigenvalues of Gaussian circulant matrices} Suppose $\{a_l\}_{ l \ge 0}$ are independent, mean zero and variance one random variables. Fix $n$.
For $1 \le t < n$, let us split
$\lambda_t$ into real and complex parts as
$\lambda_t = a_{t,n} + i b_{t, n}$, that is,
\begin{equation}\label{eq:atnbtn}
a_{t, n}= \sum\limits_{l=0}^{n-1} a_{l} \cos \left( \frac{2 \pi t l}{n}\right),\ \  b_{ t,n}=\sum\limits_{l=0}^{n-1} a_{l} \sin \left( \frac{2 \pi t l}{n}\right).\end{equation}
\begin{align}\label{eq:ortho}
 \sum_{l=0}^{n-1} \cos \left( \frac{2 \pi t l}{n}\right) \sin \left( \frac{2 \pi t' l}{n}\right)=0, \ \forall t, t'
&\text{ and } \ \sum_{l=0}^{n-1} \cos^2 \left( \frac{2 \pi t l}{n}\right)=\sum_{l=0}^{n-1} \sin^2 \left( \frac{2 \pi t l}{n}\right)= n/2 \quad \forall 0< t < n .\\
\sum_{l=0}^{n-1} \cos \left( \frac{2 \pi t l}{n}\right) \cos \left( \frac{2 \pi t' l}{n}\right)=0,
\ &\text{ and }  \sum_{l=0}^{n-1} \sin \left( \frac{2 \pi t l}{n}\right) \sin \left( \frac{2 \pi t'
l}{n}\right)=0 \quad \forall t \ne t' \ \  (\mod n).
\end{align}
For $z \in \mathbb C$, by  $\bar z$ we mean, as usual, the complex conjugate of $z$. For all $0< t, t' < n$, the following identities can easily be verified  using the above orthogonality relations
$$\E( a_{t,n}b_{t,n}) = 0, \ \ \text{and} \ \   \E  (a_{t,n}^2 )= \E(b_{t,n}^2 ) = n/2,$$
$$\bar \lambda_t = \lambda_{n-t}, \  \ \E(
\lambda_t \lambda_{t'}) = n \I( t+t'=n), \ \ \E(
|\lambda_t|^2)=n.$$
The following Lemma will be used in the proof of Theorem~\ref{theo:lsd12} and Theorem~\ref{theo:lsd45}.
\begin{lemma}\label{lem:product}
\noindent Fix $k$ and $n$. Suppose that  $\{a_l\}_{0 \le l < n}$ are i.i.d. standard normal random variables. Recall the notations $\pcal_j$ and $\Pi_j$ from Section~\ref{section:eigenvalues}. Then \\

\noindent(a) For every
$n$,  $n^{-1/2}a_{t, n}, n^{-1/2}b_{t, n}$, $ 0 \leq t \le n/2$ are i.i.d. normal with mean zero and variance $1/2$.
Consequently, any subcollection $\{\Pi_{j_1}, \Pi_{j_2}, \ldots\}$ of $\{\Pi_j\}_{0 \le j < \ell}$, so that no member of the corresponding partition blocks
$\{\pcal_{j_1}, \pcal_{j_2}, \ldots\}$ is a conjugate of any other,  are mutually independent.
 \\

\noindent(b) Suppose $1 \le j < \ell$ and $\pcal_j \cap (n - \pcal_j) =\emptyset$. Then all
$n^{-n_j/2}\Pi_j$ are distributed as $n_j$-fold product of i.i.d.\ random
variables,  each of which is distributed as $E^{1/2} U$ where $E$ and $U$ are independent random variables, $E$ is exponential with mean one and $U$ is uniform over the unit circle in $\mathbb R^2$.\\

\noindent (c) Suppose $1 \le j < \ell$ and $\pcal_j  =  n - \pcal_j$ and $n/2 \not \in\pcal_j$.  Then $n^{-n_j/2}\Pi_j$  are distributed as $(n_j/2)$-fold product of i.i.d.\ exponential random variables  with mean one.
\end{lemma}
\begin{proof} (a) Being linear combinations of $\{a_l\}_{0 \le l < n}$, $n^{-1/2}a_{t, n}, n^{-1/2}b_{t, n}$, $ 0 \leq t \le n/2$ are all jointly Gaussian.  By (\ref{eq:ortho}),
they have mean zero, variance $1/2$ and are independent.

\noindent (b) By part (a) of the lemma,  note that $n^{-1/2} \gl_t =n^{-1/2}a_{t, n}+ i n^{-1/2}b_{t, n}$ is a complex normal random variable  with mean zero and variance $1/2$ for every $0< t < n$ and moreover, they are independent by the given restriction on $\pcal_j$. The assertion follows by the observation that such a complex normal is same as $E^{1/2} U$ in distribution.

\noindent (c) If $t \in \pcal_j$ then  $n-t \in \pcal_j$ too and $t \ne n - t$. Thus $n^{-1} \gl_t \gl_{n-t} = n^{-1}( a_{t, n}^2 + b^2_{t, n})$ which, by part (a),  is distributed as $Y/2$ where $Y$ is Chi-square with two degrees of freedom.  Note that $Y/2$ has the same distribution as that of exponential random variable with mean one. The proof is complete once we observe that  $n_j$ is necessarily even and the $\gl_t$'s associated with $\pcal_j$ can be grouped into $n_j/2$ disjoint pairs like above which are mutually independent.
\end{proof}
\subsection{Proof of Theorem~\ref{thm:degenerate} }

Recall the notation $\lambda_j,  \ell, \pcal_j, n_j$ and $g_x$ from Section \ref{section:eigenvalues}. By Theorem \ref{theo:formula}, the eigenvalues of $n^{-1/2}A_{k,n}$ are given by
\[ \exp \Big(\frac{2 \pi i s}{ n_j} \Big) \times  \Big(\prod_{t \in  \pcal_j} |n^{-1/2}\lambda_{t}| \Big) ^{1/n_j}, 1\le s \le n_j, \ 0 \le j < \ell,  \]
where $i = \sqrt {-1}$.
Fix any $\eps>0$ and $0< \gth_1< \gth_2< 2\pi$.
Define
\[ B(\gth_1,\gth_2, \eps) = \Big\{(x, y) \in \mathbb R^2: r -\eps <\sqrt{ x^2 +y^2} < r+\eps, \tan^{-1} (y/x) \in [\gth_1, \gth_2] \Big\}. \]
Clearly, it is enough to prove that as $n \to \infty$,
\begin{equation}\label{eq:main_th2}
 \frac 1n \sum_{ j=0}^{\ell -1} \sum_{ s = 1}^{n_j} \I \left( \exp \Big(\frac{2 \pi i s}{ n_j} \Big) \times  \Big(\prod_{t \in  \pcal_j} |n^{-1/2}\lambda_{t}| \Big) ^{1/n_j}  \in B(\gth_1,\gth_2, \eps) \right)\stackrel{P}{\to} \frac{(\gth_2 - \gth_1)}{2\pi}.
 \end{equation}
 Note that for a fixed positive integer $C$, we have
\begin{align*}
n^{-1}\sum_{ 1 \le j < \ell: n_j \le C}  n_j & \le n^{-1} \sum_{u = 2}^C \# \big \{ 1 \le x < n: g_x = u \big \}\\
&\le  n^{-1} \sum_{u = 2}^C\# \big \{ 1 \le x < n: x k^u = x\mbox{ mod } n \big \} \\
&=  n^{-1} \sum_{u = 2}^C \# \big \{ 1 \le x < n: x (k^u- 1) = s n  \text{ for some } s \ge 1\big \} \\
& \le  n^{-1} \sum_{u = 2}^C (k^u  - 1) \le n^{-1} Ck^C   \to 0, \text{as}\ \  n \to \infty.
\end{align*}
Therefore, if we define \[ N_C = \sum_{  j =0 : \ n_j \le C}^{\ell -1}  n_j ,\] then  the above result combined with the fact that $\pcal_0 = \{0\}$ yields $N_C/n \to 0$. With $C> (2\pi)/ (\gth_2 - \gth_1)$,  the left
side of \eqref{eq:main_th2} can rewritten as
\begin{align} \label{eq:th2_intermediate}
&\frac 1n \sum_{ j=0}^{\ell -1}  \# \left \{s :\frac{2 \pi s}{ n_j}   \in [\gth_1, \gth_2] , s = 1, 2 , \ldots, n_j \right \} \times \I \left(   \Big(\prod_{t \in  \pcal_j} |n^{-1/2}\lambda_{t}| \Big) ^{1/n_j}  \in (r- \eps, r+\eps)  \right) \notag \\
= &  \frac{n-N_C}{n} \frac 1{n - N_C} \sum_{ j=0, \ n_j >C}^{\ell -1} n_j \times   n_j^{-1} \# \left \{s :\frac{ s}{ n_j}   \in (2 \pi )^{-1}[\gth_1, \gth_2] , s = 1, \ldots, n_j \right \}  \notag \\
& \hspace{4cm} \times \I \left(   \Big(\prod_{t \in  \pcal_j} |n^{-1/2}\lambda_{t}| \Big) ^{1/n_j}  \in (r- \eps, r+\eps)  \right) + O\left (\frac{N_C}{n}\right) \notag\\
= &  \frac 1{n - N_C} \sum_{ j=0, \ n_j >C}^{\ell -1} n_j \times  \left ( \frac{(\gth_2 - \gth_1)}{2\pi} + O(C^{-1})
\right) \times \I \left(   \Big(\prod_{t \in  \pcal_j} |n^{-1/2}\lambda_{t}| \Big) ^{1/n_j}  \in (r- \eps, r+\eps)  \right) + O\left (\frac{N_C}{n}\right) \notag \\
=&  \frac 1{n - N_C} \sum_{ j=0, \ n_j >C}^{\ell -1} n_j \times  \frac{(\gth_2 - \gth_1)}{2\pi}  \times \I \left(   \Big(\prod_{t \in  \pcal_j} |n^{-1/2}\lambda_{t}| \Big) ^{1/n_j}  \in (r- \eps, r+\eps)  \right) + O(C^{-1})+O\left (\frac{N_C}{n}\right)
\notag\\
= &    \frac{(\gth_2 - \gth_1)}{2\pi}  +  \frac 1{n - N_C} \sum_{ j=0, \ n_j >C}^{\ell -1} n_j \times  \I \left(   \Big(\prod_{t \in  \pcal_j} |n^{-1/2}\lambda_{t}| \Big) ^{1/n_j}  \not \in (r- \eps, r+\eps)  \right) + O(C^{-1})+O\left (\frac{N_C}{n}\right).
\end{align}
To show that the second term in the above expression converges to zero in $L^1$, hence in probability, it remains to prove,
\begin{equation}  \label{eq:prodexpC}
 \prob \left(   \Big(\prod_{t \in  \pcal_j} |n^{-1/2}\lambda_{t}| \Big) ^{1/n_j}  \not \in (r- \eps, r+\eps)  \right)
\end{equation}
is uniformly small for all $j$ such that $n_j>C$ and for all but finitely many $n$  if we take $C$ sufficiently large.

By Lemma~\ref{lem:product}, for each $1 \le t < n$,   $|n^{-1/2}\lambda_t|^2 $  is  an exponential random variable with mean one, and
$\lambda_t$ is independent of
$\lambda_{t'}$ if $t' \ne n-t$ and $|\lambda_t| = |\lambda_{t'}|$ otherwise.
 Let $ E, E_1, E_2, \ldots$ be i.i.d.\  exponential random variables with mean one.  Observe that depending or whether or not $\pcal_j$ is conjugate to itself, \eqref{eq:prodexpC} equals respectively,
\[  \prob \left(   \Big(\prod_{t=1}^{n_j/2} E_t \Big) ^{1/n_j}  \not \in (r- \eps, r+\eps)  \right)
\ \ \text{or} \ \ \prob \left(   \Big(\prod_{t=1}^{n_j} \sqrt E_t \Big) ^{1/n_j}  \not \in (r- \eps, r+\eps)  \right).\]
The theorem now follows by letting first $n \to \infty $ and then $C \to \infty$ in \eqref{eq:th2_intermediate} and by observing that Strong Law of Large Numbers implies
that \[ \left( \prod_{t=1}^{C} \sqrt E_t \right)^{1/C} \to r =  \exp( \E [ \log \sqrt E] ) \quad \text{almost surely,  \ \ \ as } C \to \infty. \]
\hfill $\square$

\subsection{Invariance Principle}
For a set $B \subseteq \mathbb R^d, d \ge 1$, let $(\partial B)^{\eta}$ denote the `$\eta$-boundary' of the set $B$, that is,  $(\partial B)^\eta := \big \{ y \in \mathbb R^d: \|y - z\|\leq \eta  \text{ for some } z \in \partial B \big \}$.
By  $\Phi_{d}(\cdot)$ we always mean the probability distribution  of a $d$-dimensional standard normal vector. We drop the subscript $1$ and write just $\Phi(\cdot)$ to denote the distribution of a standard normal random variable.

The proof of the following Lemma follows easily from Theorem 18.1, page 181 of Bhattacharya and
Ranga Rao (1976)\cite{Ranga76}. We omit the proof.
\begin{lemma}\label{rangarao}
Let $X_1,\ldots,X_{m}$ be  $\mathbb R^d$-valued  independent, mean zero random vectors
and let
$V_m={m}^{-1}\sum_{j=1}^{m} \mathrm{Cov} (X_j)$ be positive-definite.
Let $G_m$ be the
distribution of ${m}^{-1/2}T_m (X_1+X_2+ \cdots+X_m)$, where $T_m$ is the
symmetric, positive-definite matrix satisfying $T_m^2=V_m^{-1}$. If for some $\delta>0$,
$\E \|X_j\|^{(2+\delta)}<\infty$ for each $1 \le j \le m$, then there exist  constants $C_i =C_i(d)$, $i=1, 2$  such
that for any Borel set $A \subseteq \mathbb R^{d}$,
\begin{eqnarray*}
|G_m(A)-\Phi_{d}(A)|&\leq & C_1 m^{-\delta/2}\Big(m^{-1} \sum_{j=1}^{m} \E\|
T_mX_j\|^{(2+\delta)}\Big)+2 \sup_{ y \in \mathbb R^d} \Phi_d \Big ( (\partial A)^{\eta}
- y \Big),\\
&\leq & C_1 m^{-\delta/2}(\lambda _{\min}(V_m))^{-(2+\delta)}\rho_{2+\delta}+2 \sup_{ y \in \mathbb
R^d} \Phi_d\Big ( (\partial A)^{\eta}
- y \Big),
\end{eqnarray*}
where $\rho_{2+\delta}=m^{-1} \sum_{j=1}^{m} \E\| X_j\| ^{(2+\delta)}$, $\lambda _{\min}(V_m)>0$ is the smallest eigenvalue of $V_m$,  and $ \eta =
C_2 \rho_{2+\delta} n^{-\delta/2}$.
\end{lemma}

\subsection{Proof of Theorem~\ref{theo:lsd12}}
Since $\gcd(k,n)=1$,
$n^{\prime}=n$ in
Theorem \ref{theo:formula} and hence there are no zero eigenvalues.
By Lemma \ref{lem:fkn} (i),
$\upsilon_{k,n}/n \to 0$ and hence the corresponding eigenvalues do not
contribute to the LSD.
It remains to consider only the eigenvalues corresponding to the sets
$\pcal_j$ of size {\it exactly} equal to $g_1$.
From Lemma \ref{lem:fkn}(i), $ g_1  =2g$ for $n$ sufficiently large.

Recall the quantities $n_j = \#\pcal_j$, $\Pi_j = \Pi_{t \in \pcal_l} \lambda_{t}$, where $\lambda_t
= \sum\limits_{l=0}^{n- 1} a_{\ell} \omega^{tl}$, $0 \le j < n$. Also, for every integer $t \ge 0$,
$tk^g = (-1+ sn) t  = -t \mod  n$, so that, $\lambda_t$ and $\lambda_{n-t}$ belong to same
partition block $S(t) = S(n-t)$. Thus each $\Pi_j$ is a nonnegative real number. Let us define
\[ J_n  = \{ 0 \le j < \ell : \#\pcal_j  = 2g \},\]
 so that $n = 2g \#J_n+n\upsilon_{k,n}$. Since, $\upsilon_{k,n} \to 0$,
$(\#J_n)^{-1}n \to 2g$. Without any loss, we denote the index set of such $j$ as
$J_n=\{1, 2, \ldots \#J_n\}$.

Let $1, \vr, \vr^2, \ldots \vr^{2g-1}$ be all the $(2g)$th roots of unity.
Since $n_j=2g$ for every $j \in J_n$, the eigenvalues
corresponding to the set $\pcal_j$ are: \[\Pi_j^{1/2g}, \Pi_j^{1/2g}
\vr, \ldots \Pi_j^{1/2g}\vr^{2g-1}.\] Hence, it suffices to consider only the empirical distribution of
$\Pi_j^{1/2g}$ as $j$ varies over the index set $J_n$: if this sequence of empirical distributions has a limiting distribution $F$, say, then the LSD of the original
sequence $n^{-1/2}A_{k,n} $ will be $(r, \gth)$ in polar coordinates where $r$ is distributed according to  $F$ and
$\gth$ is distributed uniformly across all the $2g$ roots of unity and $r$ and $\gth$ are
independent. With this in mind, and remembering the scaling $\sqrt n$, we consider
$$F_{n}(x) =(\#J_n)^{-1}\sum_{j=1}^{\#J_{n}}
\I\left( \left[n^{-g}\Pi_j\right]^{\frac{1}{2g}}\leq x \right).$$ Since the set of $\lambda$
values corresponding to any $\pcal_j$ is closed under conjugation,
there exists  a set $\acal_j \subset \pcal_j$ of size
$g$ such that
\begin{equation*}
\pcal_j = \{ x : x \in \acal_j \text{ or } n- x \in \acal_j \}.
\end{equation*}
Combining each $\lambda_t$ with its conjugate, and recalling
the definition of $\{a_{t, n}\}$ and $\{b_{t, n}\}$ in
(\ref{eq:atnbtn}), we may write $\Pi_j$ as
$$\Pi_j=\prod_{t \in \acal_j } ( a_{t, n}^2+ b_{t, n}^2).$$
First assume the random variables $\{a_l\}_{ l \ge 0}$ are i.i.d.\ standard normal. Then by Lemma~\ref{lem:product}(c),
$F_{n}$ is the usual
 empirical distribution of $\#J_n$ observations on $(\prod_{j=1}^g E_j)^{1/2g}$ where $\{E_j\}_{1 \le j \le g}$ are
i.i.d.\ exponentials with mean one. Thus by Glivenko-Cantelli Lemma, this converges to the distribution of
$(\prod_{j=1}^g E_j)^{1/2g}$.
Though the variables involved in the empirical distribution form a triangular sequence,
the convergence is still almost sure due to the specific bounded nature of the indicator functions involved. This may be proved easily by applying Hoeffding's inequality
 and Borel-Cantelli lemma.

As mentioned earlier,  all eigenvalues corresponding to any partition block $\pcal_j$ are all the $(2g)$th roots of the product $\Pi_j$. Thus, the limit claimed in the statement of the theorem holds. So we have proved the result when the random variables $\{a_l\}_{l \ge0}$ are i.i.d. standard normal.

Now suppose that the variables $\{a_l\}_{l \ge0}$ are not necessarily normal. This case is tackled by normal approximation arguments similar to Bose and Mitra (2002)\cite{Bosemitra02} who deal with the case $k=n-1$ (and hence $g=1$). We now sketch some of the main steps.

The basic idea remains the same but in this general case,  a  technical complication arises as we need to control the Gaussian measure of the $\eta$-boundaries of some non-convex sets once we apply the invariance lemma (Lemma~\ref{rangarao}). We overcome this difficulty by suitable compactness argument.

We start by defining $$F(x)=\prob\left( \big(  \prod_{j=1}^g E_j \big)^{1/2g}\leq x  \right), \ \ x \in \mathbb R.$$
To show that the ESD converges to the required LSD in probability, we show that for every $x>0$, $$\E[F_{n}(x)] \rightarrow F(x)\ \  \mbox {and}\ \
\var[F_{n}(x)]\rightarrow 0.
$$
Note that for $x > 0$,
$$\E[F_{n}(x)]
= (\# J_n)^{-1} \sum_{j=1}^{\# J_n}\prob\big(n^{-g}\Pi_j \leq
x^{2g}\big).
$$
Lemma \ref{lem:product} motivates using normal approximations.
Towards using Lemma~\ref{rangarao}, define $2g$ dimensional random vectors
$$X_{l, j}= 2^{1/2} \left(a_{l} \cos \left( \frac{2 \pi t l}{n}\right), \ \  a_{l} \sin \left( \frac{2 \pi t l}{n}\right):\ \  t \in \acal_j \right) \quad  0 \le l < n , 1 \le j \le \# J_n.$$
Note that $$\E(X_{l,j})=0 \ \ \text{and} \ \ n^{-1} \sum_{l=1}^{n-1}\Cov (X_{l,j})=I_{2g} \ \ \forall \ \
l, \  j.$$ Fix $x>0$. Define the set $A \subseteq \mathbb R^{2g}$ as
$$A:= \Big \{ (x_j, y_j : 1 \le j \le g): \prod_{j=1}^g \big [2^{-1}( x_j^2 + y_j^2) \big ]\leq x^{2g} \Big \}.$$
Note that
$$\Big \{n^{-g}\Pi_j \leq x^{2g} \Big\} = \Big \{
        n^{-1/2}\sum_{l=0}^{n-1} X_{l,j}\in A \Big\}.$$
We want to prove
$$\E[F_n(x)]- \Phi_{2g}(A) =(\#J_n)^{-1} \sum_{l=1}^{\#J_n}
          \left(  \prob\big ( n^{-g} \Pi_j \leq x^{2g}\big )
     - \Phi_{2g}(A) \right) \rightarrow 0.$$
For that, it suffices to show that for every $\eps>0$  there exists $N = N(\eps)$ such that  for all $n \ge N$,
\[ \sup_{ j \in J_n} \left |\prob \Big (
        n^{-1/2}\sum_{l=0}^{n-1} X_{l,j}\in A \Big)
-\Phi_{2g}(A) \right |\leq \eps.\]
Fix $\eps>0$.  Find $M_1>0$ large such that $\Phi([-M_1, M_1]^c) \le \eps/(8g)$.
By Assumption \texttt{I}, $\E(n^{-1/2} a_{t, n})^2 = \E(n^{-1/2} b_{t, n})^2 = 1/2$ for any $ n \ge 1$ and $0 < t < n$. Now by Chebyshev bound, we can find $M_2>0$ such that for each $n \ge 1$ and for each $0 < t< n$,
\[  \prob ( |n^{-1/2} a_{t, n}| \ge M_2) \le \eps/(8g)  \quad \text{ and }  \ \  \prob ( |n^{-1/2} b_{t, n}| \ge M_2) \le \eps/(8g).\]
Set $M = \max \{ M_1, M_2\}$. Define the set $B := \Big \{ (x_j, y_j: 1 \le j \le g) \in \mathbb R^{2g}:  |x_j |, |y_j| \le M \ \ \forall j  \Big\}$. Then for all sufficiently large $n$,
\begin{align*}
\left |\prob \Big (
        n^{-1/2}\sum_{l=0}^{n-1} X_{l,j}\in A  \Big) - \Phi_{2g}(A )\right |
\le
\left | \prob \Big (
        n^{-1/2}\sum_{l=0}^{n-1} X_{l,j}\in A \cap B \Big) - \Phi_{2g}(A \cap B ) \right | + \eps/2.
  \end{align*}
We now apply Lemma~\ref{rangarao} for $A \cap B$ to obtain
\[ \left | \prob \Big (
        n^{-1/2}\sum_{l=0}^{n-1} X_{l,j}\in A \cap B \Big) - \Phi_{2g}(A \cap B ) \right | \le C_1 n^{-\delta/2}\rho_{2+\delta} + 2 \sup_{z \in \mathbb R^{2g}} \Phi_{2g} \Big(  (\partial (A \cap B) )^\eta - z  \Big) \]
where
$$\rho_{2+\delta} = \rho_{2+\delta} =\sup_{ 0 \le l < n , j \in J_n} n^{-1}\sum_{l=0}^{n-1} \E\|X_{l,j}\|^{2+\delta}  \quad \text{ and } \ \  \eta = \eta(n) = C_2 \rho_{2+\delta}  n^{-\delta/2}.$$
Note that $ \rho_{2+\delta} $ is uniformly bounded in $n$ by Assumption \texttt{I}.

It thus remains to show that
 \[  \sup_{z \in \mathbb R^{2g}} \Phi_{2g} \Big(  (\partial (A \cap B) )^\eta - z  \Big) \le \eps/8\]
 for all sufficiently large $n$. Note that
  \begin{align*}
 \sup_{z \in \mathbb R^{2g}} \Phi_{2g} \Big(  (\partial (A \cap B) )^\eta - z  \Big) &\le  \sup_{z \in \mathbb R^{2g}}  \int_{ (\partial (A \cap B) )^\eta} \phi(x_1 - z_1) \ldots \phi(y_{g} -z_{2g} ) dx_1 \ldots dy_{g} \\
 &\le \int_{ (\partial (A \cap B) )^\eta} dx_1 \ldots dy_{g}.
   \end{align*}
 Finally note that $\partial (A \cap B) $ is a {\em compact} $(2g-1)$-dimensional manifold which has  zero measure under the $2g$-dimensional Lebesgue measure. By compactness of $\partial (A \cap B) $, we have
\[ (\partial (A \cap B))^{\eta} \downarrow \partial (A \cap B)  \qquad \text{ as } \eta \to 0,  \]
and the claim follows by Dominated Convergence Theorem.

This proves that for $x >0$, \ $ \E[F_{n}(x)] \rightarrow F(x)$.
To show  that $\Var[F_{n}(x)]\rightarrow 0$, since the variables
involved are all bounded, it is enough to show that
$$n^{-2} \sum_{j\neq j^\prime} \Cov \left ( \I \big ( n^{-g} \Pi_j \leq x^{2g}\big ), \  \I \big ( n^{-g}\Pi_{j^\prime} \leq x^{2g}\big )\right )
\rightarrow 0.$$ Along the lines of the proof used to show $\E[F_{n}(x)] \rightarrow F(x)$, one may
now extend the vectors with $2g$ coordinates defined above to ones with $4g$ coordinates and
proceed exactly as above to verify this. We omit the routine details. This completes the proof of
Theorem~\ref{theo:lsd12}.\\  \hfill $\square$
\begin{remark} In view of Theorem \ref{theo:formula}, the above theorem can easily  extended to yield an LSD  has some positive mass at the origin. For example,  fix $g>1$ and a positive integer $m$. Also, fix $m$ primes $q_1, q_2, \ldots, q_m$ and $m$ positive integers $\beta_1, \beta_2, \ldots, \beta_m$.  Suppose the sequences $k$ and $n$ tends to infinity such  that
\begin{itemize}
\item[(i)] $k = q_1q_2 \ldots q_m   \hat k $ and $n = q_1^{\gb_1} q_2^{\gb_2} \ldots q_m^{\gb_m}\hat n$ \ with $\hat k$ and  $\hat n \to \infty$,
\item[(ii)] $k^g  = -1+ s \hat n$ where $s = o(\hat n^{p_1-1}) = o(n^{p_1-1})$ where $p_1$ is the smallest prime divisor of $g$.
\end{itemize}
 Then  $F_{{ n}^{-1/2}A_{k, n}}$
 converges weakly in probability to the distribution  which has
$1-  \Pi_{j=1}^{m} q_j^{-\beta_j}$ mass at zero, and the rest of the probability mass
is distributed as $U_1(\prod_{j=1}^g E_j)^{1/2g}$ where $U_1$ and $\{ E_j\}_{ 1 \le j \le g}$ are as in Theorem \ref{theo:lsd12}.
\end{remark}
\subsection{Proof of Theorem~\ref{theo:lsd45}}
We will not present here the detailed proof of  Theorem \ref{theo:lsd45} but let us sketch the main idea. First of all, note that $\gcd(k, n) =1$ under the given hypothesis. When $g=1$, then $k=1$ and the eigenvalue partition is the trivial partition which consists of only singletons and  clearly the partition sets $\pcal_j$, unlike the previous theorem, are not self-conjugate.

For $g \ge 2$, by Lemma~\ref{lem:fkn}(ii), it follows that
$g_1 = g$ for $n$ sufficiently large and $\upsilon_{k,n} \to 0$.
In this case also,  the partition sets $\pcal_j$ are not necessarily self-conjugate.
Indeed we will show that the number of indices $j$ such that $\pcal_j$ is self-conjugate is
asymptotically negligible compared to $n$. For that, we need to bound the cardinality of the following sets for $ 1 \le b
< g_1=g$,
\begin{align*}
 W_b := \Big\{ 0< t <n: tk^b = - t\mod n \Big \}   = \Big \{ 0< t <n : n| t(k^b+1) \Big \}.
\end{align*}
Note that $t_0(b) := n/\gcd(n, k^b+1)$ is the minimum element of $W_b$ and  every other element of the set $W_b$ is a multiple of  $t_0(b)$.
 Thus the cardinality of the set $W_b$ can be   bounded by
\[ \#W_b \le n/ t_0(b)  = \gcd( n, k^b+1).\]
Let us now estimate $\gcd( n, k^b+1)$. For $ 1 \le b < g$,
\begin{align*}
\gcd( n, k^b+1) \le \gcd( k^g - 1, k^b+1)  \le k^{\gcd(g, b)} +1 \le k^{g/p_1} +1 = (1+sn)^{1/p_1}+1 = o(n),
\end{align*}
which implies
\begin{align*}
n^{-1} \sum \limits_{ 1 \le b < g} \# W_b= o(1)
\end{align*}
as desired. So, we can ignore the partition sets which are self-conjugate.

Let  $J_n$ denote the set of all those indices $j$ for which $\#\pcal_j  = g$  and $\pcal_j  \cap (n - \pcal_j) =\emptyset$. Without loss, we assume that $J_n = \{ 1, 2, \ldots, \#J_n\}$.

Let $1, \vr, \vr^2, \ldots \vr^{g-1}$ be all the $g$th roots of unity.
The eigenvalues
corresponding to the set $\pcal_j, j \in J_n$ are: \[\Pi_j^{1/g}, \Pi_j^{1/g}
\vr, \ldots \Pi_j^{1/g}\vr^{g-1}.\]
For $j \in J_n$, unlike the previous theorem  $\Pi_j=\prod_{t \in \pcal_j } ( a_{t, n}+  i b_{t, n})$ will be complex.

Hence, we need to  consider  the empirical distribution:
\[ G_n(x, y) = \frac{1}{ g  \#J_n} \sum_{j=1}^{\# J_n} \sum_{r=1}^g \I \left( \Pi_j^{1/g} \vr^{r-1} \le x+ iy \right), \quad x, y \in \mathbb R, \]
where for  two complex numbers $w = x_1 + iy_1$ and $z = x_2 + iy_2$, by $w \le z$, we mean $x_1 \le x_2$ and $x_2 \le y_2$.

 If $\{a_l\}_{ l \ge 0}$ are i.i.d.\  $N(0,1)$, by Lemma~\ref{lem:product},  $\Pi_j^{1/g}, j \in \pcal_j $ are independent and each of them is distributed as $\Big (\prod_{t=1}^gE_t \Big)^{1/2g} U_2$ as given in the statement of the theorem.  This coupled with the fact that  $\vr^{r-1}U_2$ has the same distribution as that of $U_2$ for each $1 \le r \le g$ implies that  $\{G_n \}_{n \ge 1}$ converges to the desired LSD (say $G$) as described in the theorem.

When $\{a_l\}_{ l \ge 0}$ are not necessarily normals but only satisfy Assumption \texttt{I}, we show that $\E G_n (x, y) \to G(x, y)$ and $\var(G_n(x, y)) \to 0$  using the same line of argument as given in the proof of Theorem ~\ref{theo:lsd12}.
For that,  we again define $2g$-dimensional random vectors,
$$X_{l, j}= 2^{1/2} \left(a_{l} \cos \left( \frac{2 \pi t l}{n}\right), \ \  a_{l} \sin \left( \frac{2 \pi t l}{n}\right):\ \  t \in \pcal_j \right) \quad  0 \le l < n , 1 \le j \le \# J_n,$$
which satisfy $$\E(X_{l,j})=0 \ \ \text{and} \ \ n^{-1} \sum_{l=1}^{n-1}\Cov (X_{l,j})=I_{2g} \ \ \forall \ \
l, \  j.$$ Fix $x, y \in \mathbb R$. Define the set $A \subseteq \mathbb R^{2g}$ as
$$A:= \left \{ (x_j, y_j : 1 \le j \le g): \left( \prod_{j=1}^g \big [2^{-1/2}( x_j + iy_j) \big ] \right)^{1/g} \leq x + iy \right \}$$
so that
$$\Big \{ \Pi_j^{1/g}\vr^{r-1}  \leq x +iy \Big\} = \Big \{
        n^{-1/2}\sum_{l=0}^{n-1} X_{l,j}\in \vr^{g+1-r} A \Big\}.$$
The rest of the proof can be completed following the proof of Theorem ~\ref{theo:lsd12}, once we realize that for each $1 \le r \le g$, $\partial (\vr^{g+1-r} A)$ is again a $(2g-1)$-dimensional manifold  which has  zero measure under the $2g$-dimensional Lebesgue measure.
\hfill $\square$
\section{Proof of Theorem 5}\label{sec:spectralnorm}
We start by defining the gumbel distribution of parameter $\gth > 0$.
\begin{definition}
A probability distribution is said to be Gumbel with parameter $\gth > 0$ if its cumulative distribution function is given by
\[ \Lambda_{\gth} (x)=\exp\{-\theta \exp (- x) \},  \ \ x \in \mathbb R. \]
The special case when $\gth =1$ \ is known as standard Gumbel distribution and its cumulative distribution function is simply denoted by $\Lambda (\cdot)$.
\end{definition}
\begin{lemma}  \label{lem:maxima}
Let $E_1$ and $E_2$ be i.i.d. exponential random variables with mean one. Then

\noindent (i)  \begin{equation} \label{eq:m1} \oa K(x)  := \prob \big( E_1E_2 > x \big)=
\int_0^\infty\exp(-y)\exp(-xy^{-1})dy
  \asymp \pi^{1/2} x^{1/4} \exp(-2x^{1/2})
\end{equation}
as $x \to \infty$.\\

\noindent (ii) Let  $G$ be the distribution of $(E_1E_2)^{1/4}$.  If $G_t$ are i.i.d.\ random variables with the distribution $G$, and
$G^{(n)}:=\max \big \{G_t: 1 \leq t \leq n \big\}$, then
\[  \frac{G^{(n)}-d_n}{c_n} \stackrel{\mathcal{D}}{\rightarrow} \Lambda_{}.\]
where $c_n$ and $d_n$ are normalising constants which can be taken as follows
\begin{equation}\label{eq:normalize}
c_n =(8\log n)^{-1/2}\ \ \text{and} \ \ d_n=\frac{ (\log n)^{1/2} }{\sqrt 2}\left
(1+\frac{1}{4}\frac{\log\log n}{\log n}\right)+\frac{1}{2(8\log n)^{1/2}} \log \frac{\pi}{2}.
\end{equation}
\end{lemma}
\begin{proof} (i)
Differentiating \eqref{eq:m1}  twice, we get
\begin{equation} \label{eq:m2}
\frac{d^2}{dx^2}\oa K(x)=\int_0^\infty y^{-2}\exp(-y)\exp(-xy^{-1})dy,
\end{equation}
which implies that $\oa K$ satisfies the differential equation
\begin{align}\label{eq:m3}
x\frac{d^2}{dx^2}\oa K(x) -\oa K(x)  &=  - \int_0^{\infty} (1 - xy^{-2})\exp\bigl( - (y+ xy^{-1} )\bigr)dy \notag \\
&= \exp\bigl( - (y+ xy^{-1} )\bigr) \Big|_0^{\infty} =0, \ \ \text{for } x>0,
\end{align}
with the boundary conditions $\oa K(0)=1$ and $\oa K(\infty)=0$.

From the theory of second order differential equations
the only solution to
\eqref{eq:m3} is
$$\oa K(x)= \pi x^{1/2}H_1^1(2ix^{1/2}), \ \  x> 0$$
where $i^2=-1$. The function $H_1^1$ is given by
(see
Watson (1944)\cite{Watson44})
$$H_1^1(x)= J_1(x)+iY_1(x)$$
where
$J_1$ and $Y_1$ are order one  Bessel functions of the first and second kind respectively.

It also follows from the
theory of the asymptotic properties of the Bessel functions
$J_1$ and $Y_1$, that \begin{equation}\label{eq:ktail} \oa K(x) \asymp \pi^{1/2} x^{1/4}
\exp(-2x^{1/2}) \ \ \text{as} \ \ x \to \infty.
\end{equation}
\noindent  (ii)  Now from (\ref{eq:ktail}),  \begin{equation}\label{eq:gtaildef} \oa G(x)=P\{ (E_1E_2)^{1/4}
> x\}
 \asymp \pi^{1/2} x \exp(-2x^2) \ \ \text{as} \ \  x \to
 \infty.
\end{equation}
By Proposition 1.1 and the development on pages 43 and 44 of Resnick (1996)\cite{Resnick96},
we need to show that,
$$\oa G(x)=\gth(x)(1-F_\#(x))
\ \ \text{where}
\lim_{x\to \infty} \gth(x)=\gth
> 0 $$
and, there exists some $x_0 $ and a function $f$ such that $f(y) > 0 $ for $y
> x_0$ and such that $f $ has an absolute continuous density with
$f^\prime(x)\to 0$ as $x \to \infty$ so that \begin{equation} 1-F_\#(x)=\exp\Bigl( {-\int_{x_0}^x
(1/f(y))dy}\Bigr), \,\ \  x> x_0.
\end{equation}
Moreover, a choice for the normalizing constants $c_n$ and $d_n$ is then given by
\begin{equation}\label{eq:defanbn}
  d^*_n=\Big( 1/(1-F_\#)\Big)^{-1}(n), \ \ c^*_n=f(d^*_n).
 \end{equation}
 Then
 \[  \frac{G^{(n)}-d^*_n}{c^*_n} \stackrel{\mathcal{D}}{\rightarrow} \Lambda_{\gth}.\]
Towards this end, define for $ x\geq 1$,
\begin{equation}\label{eq:candF}
\gth(x)\asymp \pi^{1/2}e^{-2}, \ \ \ \ 1-F_\#(x)=x\exp{\Bigl(-2(x^2-1)\Bigr)}, \ \  x \geq 1=x_0.
\end{equation}
To solve for $f$, taking $\log$ on both sides,
\begin{equation}
\log x-2(x^2-1)=-\int_1^x\frac{1}{f(y)}dy.
\end{equation}
Taking derivative,
$$\frac{1}{x}-2(2x)=-\frac{1}{f(x)}$$
or
$$f(x)=\frac{x}{4x^2-1}
\asymp \frac{1}{4x}\ \ \text{as}\ \ x\to \infty.$$
Note that $d^*_n$ (to be obtained) will tend to $\infty$ as $n \to \infty$. Hence
$$c^*_n=f(d^*_n)\asymp(4d^*_n)^{-1}.$$
We now proceed to obtain (the asymptotic form of) $d^*_n$. Using the defining equation
(\ref{eq:defanbn}),
\begin{equation}\label{eq:defbn}d^*_n\exp^{-2((d^*_n)^2-1)}=n^{-1}.
\end{equation}
Clearly, from the above,  we may write
$$d^*_n=\Big(\frac{\log n}{2}\Big )^{1/2}(1+\delta_n)$$
where $\delta_n \to 0$ is a \emph{positive} sequence to be appropriately chosen. Thus, again using
(\ref{eq:defbn}), we obtain
$$(\log n) (\delta_n^2+2\delta_n)-\big(\frac{1}{2}\log \log n+\xi_n)=0$$
where
$$\xi_n=2-\frac{1}{2}\log 2+\log (1+\delta_n).$$
``Solving" the quadratic, and then using expansion $\sqrt{1+x}=1+\frac{1}{2}x+O(x^2)$ as $x\to 0$, we easily see that
$$\delta_n=\frac{-2+\sqrt{4+4(\frac{1}{2}\log \log n +\xi_n)/\log
n}}{2}=
\frac{1}2{}\left (\frac{\frac{1}{2}\log \log n +\xi_n}{\log
n}\right)+O\left(\frac{(\log \log n)^2}{(\log n)^2}\right).$$ Hence
$$d^*_n=\big(\frac{\log
n}{2}\big)^{1/2}\left(1+ \frac{\frac{1}{2}\log\log n+\xi_n}{2\log n}\right)+O\left(\frac{(\log \log
n)^2}{(\log n)^{3/2}}\right).$$
Simplifying, and dropping appropriate small order terms, we see
that
\[  \frac{G^{(n)}-\hat d_n}{\hat c_n} \stackrel{\mathcal{D}}{\rightarrow} \Lambda_{\pi^{1/2} e^{-2}}.\]
where
\begin{eqnarray*}
\hat  c_n  =  (8 \log n)^{-1/2} \ \
\text{and} \ \
\hat d_n = \frac{ (\log
n)^{1/2} }{\sqrt 2}\left(1+\frac{1}{4}\frac{\log\log n}{\log n}\right)+ \frac{1}{(8\log n)^{1/2}}
(2-\frac{1}{2}\log 2 ).
\end{eqnarray*}
To convert the above convergence to
standard Gumbel distribution, we use
the following result of de Haan and Ferreira (2006)\cite{Haan06}[Theorem 1.1.2]  which says that the following two statements are equivalent for any
sequence of $a_n>0, b_n$ of constants and any nondegenerate distribution function $H$.

\noindent (i) For each continuity point $x$ of $H$,
\[ \lim_{n \to \infty} G^n( c_n x + d_n)  = H(x),\]
(ii) For each $x > 0$ continuity point of $H^{-1}(e^{-1/x} )$,
\[ \lim_{ t \to \infty} \frac{ \left( 1/ (1 - G) \right)^{-1}( tx) - d_{[t]} }{c_{[t]}} =
H^{-1}(e^{-1/x} ).  \] Now the relation $\gL_{\gth}^{-1}(e^{-1/x} ) - \gL^{-1}(e^{-1/x} )  = \log
\gth$ and a simple calculation yield that
\[ c_n = \hat c_n, \ \ \ d_n = \hat d_n + \hat c_n \log(\pi^{1/2} e^{-2} ).\]
\end{proof}
\subsection{Some preliminary lemmas}  First of all, note that $\gcd(k, n) =1$ and hence $n' =n$. It is easy to check that $g_1 = 4$ and
\[  \{ x \in \mathbb Z_n : g_x < g_1 \}
= \left \{ \begin{array}{cc} \{0, n/2\} & \text{ if $n$ is even} \\ \{0\} & \text{ if $n$ is odd}.
\end{array} \right.\]
Thus the eigenvalue partition of
$\{0, 1, 2, \ldots, n-1\}$ can be listed as
$\pcal_1,
\pcal_2, \ldots, \pcal_q $, each of which is of size $4$. Since each $\pcal_j, 1 \le j \le q$ is self-conjugate, we can find a set $\acal_j
\subset \pcal_j$ of size $2$ such that
\begin{equation}\label{eq:a_t}
\pcal_j = \{ x : x \in \acal_j \text{ or } n- x \in \acal_j \}.
\end{equation}
 For any sequence of random variables  $b = \{b_l\}_{ t \ge 0}$,  define
 \begin{equation} \label{eq:srformula1}
 \beta_{b, n}(j)  = n^{-2} \prod_{ t \in \acal_j} \left | \sum\limits_{l=0}^{n-1} b_{l} \omega^{ t l} \right |^2, \ \ \  \omega = \exp \left(\frac{2 \pi i }{n} \right), \quad 1 \le j \le q.
 \end{equation}
The next lemma helps us to go from bounded to unbounded entries. For each $n \ge 1$, define
a triangular array of centered random variables $\{ \bar a^{(n)}_l \}_{ 0 \le l < n} $ by
\[ \bar a_ l =    \bar a^{(n)}_l = a_l I_{|a_l| \le n^{ 1/ \gc} } - \E a_l I_{|a_l| \le n^{ 1/ \gc} }.\]
\begin{lemma}[Truncation] \label{lem:trunc}
Assume $\E | a_l|^{ \gc} < \infty$ for some $\gc > 2$. Then, almost surely,
\[ \max_{ 1 \le j \le q} (\beta_{a, n}(j) ) ^{1/4}  -  \max_{ 1 \le j \le q} (\beta_{\bar a, n}(j) ) ^{1/4} = o(1).\]
\end{lemma}
\begin{proof} Since $\sum_{l=0}^{n-1} \omega^{ t l} $ = 0 for $0 < t <n$, it follows that $\beta_{\bar a, n}(j) = \beta_{\tilde a, n}(j)$ where
\[ \tilde a_l = \tilde a^{(n)} _l = \bar a_l  + \E a_l I_{|a_l| \le n^{ 1/ \gc} }=a_l I_{|a_l| \le n^{ 1/ \gc}}.\]
By Borel-Cantelli lemma, with probability one, $ \sum_{t=0}^{\infty} |a_l| I_{|a_l | > l^{1/\gc} }$
is finite and has only finitely many non-zero terms. Thus there exists  an integer $N \ge
0$,  which may depend on the sample point, such that
\begin{align}\label{eq:trucbc}
\sum_{l=m}^{n-1}| \tilde a^{(n)}_l  - a_l | = \sum_{l=m}^{n-1} |a_l| I_{|a_l | > n^{1/\gc} } \le
\sum_{t=m}^{\infty} |a_l| I_{|a_l | > t^{1/\gc} } = \sum_{l=m}^{N} |a_l| I_{|a_l | >
l^{1/\gc} }.
\end{align}
Consequently, if $ m > N$, the left side of \eqref{eq:trucbc} is zero. Therefore, the
terms of the two sequences $\{a_l \}_{ m \le l < n} $ and $\{\tilde a^{(n)}_l\}_{ m \le l < n}$ are identical almost surely  for
all sufficiently large $n$ and the assertion follows immediately.
\end{proof}
\begin{lemma}[Bonferroni inequality] \label{bonferroni}
Let $(\Omega, \mathcal F,
\prob)$ be a probability space and let $B_1, B_2, \ldots, B_n$ be events from $\mathcal F$. Then for every integer $m \ge 1$,
\begin{equation} \label{e:bonf}
\sum_{j=1}^{2m} (-1)^{j-1} S_{j,n} \le \prob \Big( \bigcup_{j=1}^n B_i \Big ) \le \sum_{j=1}^{2m-1}
(-1)^{j-1}S_{j,n},
\end{equation}
where \[ S_{j,n}  := \sum_{ 1 \le i_1 < i_2 < \ldots < i_j \le n} \prob \Big( \bigcap_{l=1}^j
B_{i_l} \Big ). \]
\end{lemma}

\begin{lemma}\label{lem:tailestimation}
Fix $ x \in \mathbb R$. Let $E_1, E_2, c_n$ and $d_n$ be as in Lemma~\ref{lem:maxima}. Let $\sigma^2_n = n^{-c}$, $c>0$. Then there exists  some positive constant $K=K(x)$ such that
 \[ \prob \left ( (E_1 E_2)^{1/4} > (1 + \sigma^2_n)^{-1/2} (c_n x + d_n) \right) \le \frac{K}{n},\ \ x \in \mathbb R.\]
 \end{lemma}
\begin{proof}
Since $(1 +y)^{-1/2} \ge 1 - y/2$ for $y > 0$,
\[ \prob \left ( (E_1 E_2)^{1/4} > (1 + \sigma^2_n)^{-1/2} (c_n x + d_n) \right) \le \prob \left ( (E_1 E_2)^{1/4} > (1  - \sigma^2_n/2 )(c_n x + d_n) \right).\]
Recall the representation
$$ \prob( (E_1 E_2)^{1/4} >  x ) =\gth(x)(1-F_\#(x))  \text{ as }  x \to \infty.$$
Note that
 $(1 - \sigma^2_n/2) (c_n x + d_n)  = d^*_n +  (d_n - d^*_n) + c_n x-  (c_n x + d_n) \sigma^2_n/2 =
d^*_n +o_x(1)$ where we use the facts that $c_n \to  0 $, $ (d_n  - d^*_n)/ c_n  = o(1)$ and $d_n
\sim \sqrt{ \log n}$. The lemma now easily follows once we note that $ 1- F_{\#} (d^*_n) = n^{-1}$.
\end{proof}
\subsection{A strong invariance principle} We now state the normal approximation result and a suitable corollary that we need.  For $d \ge 1$,
and any distinct integers $i_1, i_2, \ldots, i_d$, from $\Big \{ 1,2, \ldots, \lceil \frac{n-1}{2}\rceil
\Big \}$, define
\[  v_{2d}(l)  = \left(  \cos\left ( \frac{ 2 \pi {i_j} l}{n} \right), \sin \left ( \frac{ 2 \pi {i_j} l}{n}\right) : 1 \le j \le d  \right)^T, \quad l \in \mathbb Z_n.\]
Let $\varphi_{\Sigma}(\cdot)$ denote the density of the $2d$-dimensional Gaussian vector having
mean zero and covariance matrix $\Sigma$ and let $I_{2d}$ be the identity matrix of order $2d$.
\begin{lemma}[Normal approximation, Davis and Mikosch (1999)\cite{Mikosch99}] \label{lm:normalapprox}
Fix $d \ge 1$ and $\gc > 2$ and let $\tilde p_n$ be the density function of
\[ 2^{1/2} n^{-1/2} \sum_{l=0}^{n-1}( \bar a_{l} + \sigma_n N_{l} ) v_{2d}(l),\]
where $\{N_l\}_{ l \ge 0}$ is a sequence of i.i.d.\ $N(0,1)$ random variables, independent of $\{a_l\}_{ l \ge 0}$ and
$\sigma_n^2 = \var(\bar a_0)s_n^2$. If $n ^{-2c} \ln n \le s_n^2 \le 1 $ with $c = 1/2  -
(1-\delta)/ \gc$  \  \ for arbitrarily small $\delta >0$, then the relation
\[ \tilde p_n(x)  = \varphi_{(1+ \sigma_n^2)I_{2d}}(x) (1 + \eps_n ) \quad \text{ with } \eps_n \to 0 \]
holds uniformly for $\|x\|^3 = o_d ( n^{1/2 - 1/ \gc}),  \ x \in \mathbb R^{2d} $.
\end{lemma}
\begin{corollary}\label{cor:approx_by_normal}
Let $\gc > 2$  and $\sigma_n^2 = n^{-c}$ where $c$ is as in Lemma \ref{lm:normalapprox}. Let $B \subseteq \mathbb R^{2d}$ be
a measurable set. Then
\[ \left | \int_{B} \tilde p_n(x) dx-  \int_{B}   \varphi_{(1+ \sigma_n^2)I_{2d}}(x) dx \right | \le  \eps_n  \int_{B}   \varphi_{(1+ \sigma_n^2)I_{2d}}(x) dx + O_d (\exp ( - n^{\eta} ) ), \]
for some $\eta > 0$ and   uniformly over all the $d$-tuples of distinct integers $ 1 \le i_1< i_2<
\ldots < i_d \le \lceil \frac{n-1}{2}\rceil$.
\end{corollary}
\begin{proof}
Set $r = n^{\ga}$ where $0<  \ga < 1/2 - 1/ \gc $. Using Lemma \ref{lm:normalapprox}, we have,
\begin{align*}
 &\left | \int_{B} \tilde p_n(x) dx-  \int_{B}   \varphi_{(1+ \sigma_n^2)I_{2d}}(x) dx \right | \\
\le&   \left | \int_{B \cap \{\|x \| \le r \}} \tilde p_n(x) dx-  \int_{B\cap \{\|x \| \le r \}}
\varphi_{(1+ \sigma_n^2)I_{2d}}(x) dx  \right| +
 \int_{B\cap \{\|x \| > r \}} \tilde p_n(x) dx +   \int_{B\cap \{\|x \| > r \}}   \varphi_{(1+ \sigma_n^2)I_{2d}}(x) dx\\
 \le&  \eps_n \int_{B \cap \{\|x \| \le r \}} \varphi_{(1+ \sigma_n^2)I_{2d}}(x) dx +
 \int_{ \{\|x \| > r \}} \tilde p_n(x) dx +   \int_{ \{\|x \| > r \} }   \varphi_{(1+ \sigma_n^2)I_{2d}}(x) dx=T_1+T_2+T_3 \ \ (say).
\end{align*}
Let $v^{(j)}_{2d}(l)$
denote the $j$-th coordinate of
$v_{2d}(l)$, $1 \le j \le 2d$. Then, using the normal tail bound,  $ \prob \big( | N(0, \sigma^2) |  > x \big ) \le 2 e^{-x/\sigma}$ for $x > 0$,
\begin{align*}
 T_2&=\int_{ \{\|x \| > r \}} \tilde p_n (x) dx = \prob \left (  \Big \| 2^{1/2} n^{-1/2} \sum_{l=0}^{n-1}( \bar a_{l} + \sigma_n N_{l} ) v_{2d}(l)  \Big \|  > r \right)\\
 &\le 2d \max_{ 1 \le j \le 2d} \prob \left (  \Big | 2^{1/2} n^{-1/2} \sum_{l=0}^{n-1}( \bar a_{l} + \sigma_n N_{l} ) v^{(j)}_d(l) \Big |  > r/(2d)      \right)\\
 &\le 2d \max_{ 1 \le j \le 2d} \prob \left (  \Big | n^{-1/2} \sum_{l=0}^{n-1} \bar a_{l}  v^{(j)}_d(l) \Big |  > r /(4\sqrt 2 d)     \right) + 4d \exp \Big( - r n^{c/2} /(4\sqrt 2d) \Big).
 \end{align*}
 Note that
 $\bar a_{l}  v^{(j)}_d(l), 0 \le l < n$ are independent, have mean zero and  variance at most one and are bounded by $2n^{1/\gc}$. Therefore, by applying Bernstein's inequality and simplifying, for some constant $K >0$,
 \begin{align*}
  \prob \left (  \Big | n^{-1/2} \sum_{l=0}^{n-1} \bar a_{l}  v^{(j)}_d(l) \Big |  > r /4\sqrt 2 d     \right)  \le \exp(  - K r^2).
 \end{align*}
Further,
\[  T_3 =\int_{ \{\|x \| > r \} }   \varphi_{(1+ \sigma_n^2)I_{2d}}(x) dx \le 4d \exp( -  r/4d). \]
Combining the above estimates finishes the proof.
\end{proof}
\subsection{Proof of Theorem~\ref{theo:max}}
First assume that $n$ is even. Then $k$ must be odd and $S(n/2) =\{ n/2\}$.
Thus with the previous notation,
\[  \spr( n^{-1/2}A_{k,n} ) = \max \left \{ \max_{ 1 \le j \le q} (\beta_{a,
n}(j) ) ^{1/4}, |n^{-1/2}\gl_0|, \ |n^{-1/2}\gl_{n/2}| \right \}.\]
Since $d_q \to \infty$ and $c_q \to 0$, by Chebyshev inequality, we have
\[ \sup_{0  \le t < n} \prob \left ( |n^{-1/2}\gl_t| \ge x c_q + d_q \right) \to 0, \quad \text{for each } x \in \mathbb R.\]
Thus  finding the limiting distribution $ \spr(n^{-1/2}A_{k,n} )$ is
asymptotically equivalent to finding the limiting distribution of $\max_{ 1 \le j \le q} (\beta_{a,n}(j) ) ^{1/4}$. Clearly,
this is also true if $n$ is odd as that case is even simpler.

 Now, as in the proof of Theorem \ref{theo:lsd12}, first assume that $\{a_l\}_{ l \ge 0}$
are i.i.d.\ standard normal. Let $\{ E_j \}_{j \ge 1}$  be i.i.d.\ standard exponentials. By Lemma~\ref{lem:product}, it easily follows that
$$\prob\Big(\max_{ 1 \le t \le q} (\beta_{a, n}(t) ) ^{1/4} > c_qx+ d_q\Big)
=\prob\Big((E_{2j-1}E_{2j})^{1/4} > c_qx+ d_q \text{ for some } 1 \le j \le q\Big).$$  The Theorem
then follows in this special case from Lemma \ref{lem:maxima}.

We now tackle the general case by using truncation of $\{a_l\}_{ l \ge 0}$, Bonferroni's inequality and the
strong normal approximation result given in the previous subsections.  Fix $x \in \mathbb R$. For notational convenience,
define
\begin{align*} Q^{(n)}_1 &:= \prob \left( \max_{ 1 \le j \le q} (\beta_{\bar a + \sigma_n N, n}(j) ) ^{1/4} > c_q x + d_q\right),\\
Q^{(n)}_2 &:=\prob \left( \max_{ 1 \le j \le q} (1+ \sigma_n^2) (E_{2j-1} E_{2j} ) ^{1/4} > c_q x + d_q \right ) ,
\end{align*}
where $\{N_l\}_{ l \ge 0}$ is a sequence of i.i.d.\ standard normals random variables. Our goal is to
approximate $ Q^{(n)}_1$ by  the simpler quantity $Q^{(n)}_2$.  By Bonferroni's inequality, for all $m \ge 1$,
\begin{equation}\label{eq:Ssandwich} \sum_{j=1}^{2m} (-1)^{j-1} S_{j,n} \le Q^{(n)}_1 \le
\sum_{j=1}^{2m-1} (-1)^{j-1}S_{j,n},
\end{equation} where \[ S_{j,n}  = \sum_{ 1 \le t_1 < t_2 <
\ldots < t_j \le q} \prob \left( (\beta_{\bar a + \sigma_n N, n}(t_1) ) ^{1/4} > c_q x + d_q,
\ldots, (\beta_{\bar a + \sigma_n N, n}(t_j) ) ^{1/4}
> c_q x + d_q \right ). \]
Similarly, we have \begin{equation}\label{eq:Tsandwich} \sum_{j=1}^{2m}
(-1)^{j-1} T_{j,n} \le Q^{(n)}_2 \le \sum_{j=1}^{2m-1} (-1)^{j-1}T_{j,n},
 \end{equation}
where
 \[ T_{j,n}  = \sum_{ 1 \le t_1 < t_2 < \ldots < t_j \le q} \prob \Big( (1+ \sigma_n^2) (E_{2t_1-1} E_{2t_1} ) ^{1/4} > c_q x + d_q, \ldots,(1+ \sigma_n^2) (E_{2t_j-1} E_{2t_j} ) ^{1/4} > c_q x + d_q \Big ). \]
Therefore, the difference between $Q^{(n)}_1$ and $Q^{(n)}_2$ can be bounded as follows:
\begin{equation}\label{eq:maindiff}
  \sum_{j=1}^{2m} (-1)^{j-1} (S_{j,n} - T_{j,n} )   - T_{2m+1, n}  \le Q^{(n)}_1  -Q^{(n)}_2 \le \sum_{j=1}^{2m-1} (-1)^{j-1}(S_{j,n} - T_{j,n}) + T_{2m, n},
  \end{equation}
for each $m \ge 1$. By independence and Lemma \ref{lem:tailestimation}, there exists $K = K(x)$ such that
\begin{equation}\label{eq:Tbound}
T_{j,n} \le { n \choose j} \frac{K^j}{n^j} \le \frac{K^j}{j!} \quad \text{for all } n, j \ge 1.
\end{equation}
Consequently, $\lim_{j \to \infty} \limsup_n T_{j,n} = 0$.

Now fix $j \ge 1$. Let us bound the
difference between $S_{j,n}$ and $T_{j,n}$.
Let $\acal_{t}$ defined in \eqref{eq:a_t} be represented as $\acal_{t}=\{ e_{t},  e'_{t} \}$.
For $1 \le t_1 < t_2 < \ldots < t_j \le q$, define
\[ v_{2j}(l) = \left ( \cos \left ( \frac{ 2 \pi l e_{t_1}  }{n} \right) ,
\sin\left ( \frac{ 2 \pi l e_{t_1}  }{n} \right) ,
 \cos \left ( \frac{ 2 \pi l e'_{t_1}  }{n} \right) ,
 \ldots,
\cos \left( \frac{ 2 \pi l e_{t_j}'  }{n} \right) ,
\sin \left ( \frac{ 2 \pi l e_{t_j}'  }{n} \right)  \right).
\]
 Then,
 \begin{align*} \prob \Big( (\beta_{\bar a + \sigma_n N, n}(t_1) ) ^{1/4} > c_q x + d_q, \ldots, (\beta_{\bar a + \sigma_n N, n}(t_j) ) ^{1/4} > c_q x + d_q \Big )\\
 = \prob \Big( 2^{1/2}n^{-1/2} \sum_{l = 0}^{n-1} (\bar a_l + \sigma_n N_l) v_{2j}(l) \in B^{(j)}_n \Big),
 \end{align*}
   where \[  B^{(j)}_n := \left \{ y \in \mathbb R^{4j}:  (y_{4t +1} ^2+ y_{4t+2}^2)^{1/4} (y_{4t +3} ^2+ y_{4t+4}^2)^{1/4} > 2^{1/2}( c_qx+ d_q), 0 \le t < j \right \}. \]
 By Corollary \ref{cor:approx_by_normal} and the fact  $N_1^2 + N_2^2 \stackrel{\mathcal{D}}{=} 2E_1$, we deduce that  uniformly over all the $d$-tuples $1 \le t_1< t_2 < \ldots < t_j \le q$,
  \begin{align*}  &\left| \prob \Big( 2^{1/2}n^{-1/2} \sum_{l = 0}^{n-1} (\bar a_l + \sigma_n N_l) v_{2j}(l) \in B^{(j)}_n \Big)  - \prob\Big( (1+ \sigma_n^2)^{1/2} (E_{2t_m -1} E_{2t_m})^{1/4} > c_qx+ d_q, 1 \le m \le j \Big) \right|\\
 &\le \eps_n \prob\Big( (1+ \sigma_n^2)^{1/2} (E_{2t_m -1} E_{2t_m})^{1/4} > c_qx+ d_q, 1 \le m \le j \Big) + O(\exp(-n^{\eta}) ).
  \end{align*}
Therefore, as $n \to \infty$,
\begin{equation}\label{eq:STdiff}
|S_{j,n} - T_{j,n}| \le \eps_n T_{j,n} + { n \choose j} O(\exp(-n^{\eta}) ) \le \eps_n
\frac{K^j}{j!} + o(1) \to 0, \end{equation}where $O(\cdot)$ and $o(\cdot)$ are uniform over $j$.
 Hence using (\ref{eq:Ssandwich}), (\ref{eq:Tsandwich}), (\ref{eq:Tbound}) and
 (\ref{eq:STdiff}),
 we have
 \[ \limsup_n |Q^{(n)}_1 - Q^{(n)}_2| \le  \limsup_n T_{2m+1, n} +  \limsup_n T_{2m,n}  \quad \text{ for each } m \ge 1. \]
 Letting $m \to \infty$, we conclude $ \lim_n Q^{(n)}_1 - Q^{(n)}_2= 0$.  Since by Lemma \ref{lem:maxima},
 \[ \displaystyle \max_{ 1 \le j \le q} (E_{2j-1} E_{2j} ) ^{1/4}  = O_p(  (\log n)^{1/2} ) \ \  \text{and}\ \
 \sigma_n^2 = n^{-c},\] it follows that
 \[ \displaystyle \frac{ (1+ \sigma_n^2)^{1/2}  \max_{ 1 \le j \le q}  (E_{2j-1} E_{2j} ) ^{1/4} -   d_q}{c_q}
 \stackrel{\mathcal{D}}{\to} \Lambda_{}\]
and consequently,
\begin{align*}
 \displaystyle \frac{  \max_{ 1 \le j \le q} (\beta_{\bar a + \sigma_n N, n}(j) ) ^{1/4} -   d_q}{c_q}
 \stackrel{\mathcal{D}}{\to} \Lambda_{}.
 \end{align*}
In view of Lemma \ref{lem:trunc}, it now suffices to show that
\[  \max_{ 1 \le j \le q} (\beta_{\bar a + \sigma_n N, n}(j) ) ^{1/4}  -   \max_{ 1 \le j \le q} (\beta_{\bar a , n}(j) )
^{1/4}   = o_p( c_q).\]
 We use the basic inequality
\[ \big| |z_1z_2|  - |w_1w_2| \big| \le \big(|z_1| + |w_2| \big) \max \Big \{ |z_1 - w_1|, |z_2 - w_2| \Big \}, \quad z_1, z_2, w_1,  w_2 \in \mathbb C, \]
to obtain
\begin{align*} \left | \max_{ 1 \le j \le q} (\beta_{\bar a + \sigma_n N, n}(j) ) ^{1/2}  -   \max_{ 1 \le j \le q} (\beta_{\bar a , n}(j) ) ^{1/2} \right |
&\le \Big (M_n( \bar a + \sigma_n N) +  M_n( \bar a) \Big )  M_n( \sigma_n N)  \\
&\le \Big (2 M_n( \bar a + \sigma_n N) +  M_n( \sigma_n N) \Big )  M_n( \sigma_n N)
\end{align*}
where,  for any sequence of random variables $X = \{X_l\}_{ l \ge 0}$,
$$M_n(X) :=  \max_{ 1 \le t \le n}  \left | n^{-1/2} \sum\limits_{l=0}^{n-1} X_{l}\omega^{ t l} \right |.$$ As a trivial consequence of Theorem 2.1 of Davis and Mikosch (1999)\cite{Mikosch99}, we have
\[M_n^2( \sigma_n N) = O_p( \sigma_n \log n)\ \ \text{and} \ \ M_n^2( \bar a + \sigma_n N)  =
O_p(\log n).\] Together with $\sigma_n = n^{-c/2}$ they imply that
\[ \max_{ 1 \le j \le q} (\beta_{\bar a + \sigma_n N,
n}(j) ) ^{1/2}  - \max_{ 1 \le j \le q} (\beta_{\bar a , n}(j) ) ^{1/2}  = o_p(n^{-c/4}).\] From
the inequality
\[ |\sqrt y_1 - \sqrt y_2| \le \frac{1}{ \min \{ \sqrt y_1, \sqrt y_2 \} } |y_1 - y_2|,  \ \ y_1, y_2 > 0\]
it easily follows that
\[  \max_{ 1 \le j \le q} (\beta_{\bar a + \sigma_n N, n}(j) ) ^{1/4}  -   \max_{ 1 \le j \le q} (\beta_{\bar a , n}(j) ) ^{1/4}  = o_p(n^{-c/8}) = o_p(c_q).\]
This completes the proof of Theorem \ref{theo:max}. \hfill $\Box$
\section{Concluding remarks and open problems}
To establish the LSD of $k$-circulants for more general subsequential choices of $(k, n)$,
a much more comprehensive
study of the orbits of the translation operator
acting on the ring $\mathbb Z_{n'}$  by $\mathbb T_k(x) = x k \mod n'$ is required.  In particular, one may perhaps first establish an
asymptotic negligibility criteria similar to
that given in Lemma~\ref{lem:fkn}. Then,
along the line similar to that of the proofs of Theorems~\ref{theo:lsd12} and \ref{theo:lsd45} - first using the abundant independence structure among the eigenvalues of the $k$-circulant matrices  when the input sequence is i.i.d.\  normals as given in Lemma~\ref{lem:product} and then claiming universality through an appropriate use of the invariance principle. What particularly complicates matters is that in general there {\it may} be contributing classes of several sizes as opposed to only one (of size $2g$ or $g$) that we saw in Theorems~\ref{theo:lsd12} and \ref{theo:lsd45} respectively. Thus it is also interesting to investigate whether we can select $k = k(n)$ in a relatively simple way so that  there  exist finitely many  positive integers $h_1, h_2, \ldots, h_r, r>1 $ with
\[ \# \big \{ x \in \mathbb Z_{n'} : g_x = h_j \big \}/n' \to c_j >0, \quad 1 \le j \le r, \]
where $c_1 + \ldots+ c_r =1$. In that case the LSD would be an attractive mixture distribution.

Establishing the limit distribution of the spectral radius
for general subsequential choices of $(k, n)$ appears to be even more challenging.
In fact, even under the set up of Theorems~\ref{theo:lsd12} and \ref{theo:lsd45}, this  seems to be a nontrivial problem.
As a first step, one needs to find max-domain of attraction for $(\prod_{j=1}^{g} E_j)^{1/2g}$  where $\{E_j\}_{ 1 \le j \le g}$ are i.i.d.\ exponentials which  requires a detailed understanding of the behaviour of $\prob(\prod_{j=1}^g E_j > x)$ as $x \to \infty$.
When $ g > 2$, we were unable to locate
any results on this.
Our preliminary
investigation shows that this is fairly involved
and we are currently working on this problem. Moreover,  an extra layer of difficulty  arises while dealing with the spectral radius  out of  the fact that  we can not immediately  ignore some `bad' classes of eigenvalues whose proportions are asymptotically negligible like  we did while establishing the LSD.

\subsection*{Acknowledgement} We are grateful to Z.\ D.\ Bai for help in accessing the
article Zhou (1996). We also thank Rajat Hazra and Koushik Saha for some interesting discussions.
We are especially thankful to Koushik Saha for a careful reading of the manuscript and pointing out
many typographical errors. We are extremely grateful to the Referee for his critical, constructive
and detailed comments on the manuscript.
\section*{Appendix}
Here we provide a proof of Theorem \ref{theo:formula}. Recall that
for any two positive integers $k$ and $n$, $p_1 <p_2<\ldots < p_c$
are all their common prime factors so that,
$$n=n^{\prime} \prod_{q=1}^{c} p_q^{\beta_q} \ \text{  and  } \ \  k=k^{\prime}
\prod_{q=1}^{c} p_q^{\alpha_q}$$ where $\alpha_q,\ \beta_q \ge 1$ and $n^{\prime}$, $k^{\prime}$,
$p_q$ are pairwise relatively prime. Define
\begin{equation}\label{eq:modrelation}
m := \max_{1 \le q \le c} \lceil \beta_q/\alpha_q \rceil, \ \ \
[t]_{m,b} :=  tk^m \mbox{ mod } b,\ b \mbox{ is a positive integer}.
\end{equation}

Let $e_{m,d}$ be a $d \times 1$ vector whose only nonzero element
is $1$ at  $(m \mbox{ mod } d)$-th position, $E_{m,d}$ be the $d
\times d$ matrix with $e_{jm,d}$, $ 0 \leq j < d$ as its columns
and for dummy symbols $\delta_0, \delta_1, \ldots$, let
$\Delta_{m,b,d}$ be a diagonal matrix as given below.
\begin{eqnarray}e_{m,d}&=&
\left[ \begin{array} {c}
            0 \\
             \vdots \\
            1\\
            \vdots
            \end{array} \right]_{d\times 1}, \\
E_{m,d} &=&
\left[ e_{0,d}\ \ e_{m,d}\ \ e_{2m,d} \ldots e_{(d-1)m,d}\right],\\
\Delta_{m,b,d} &=& \text{diag} \left[ \delta_{[0]_{m,b}},\
\delta_{[1]_{m,b}},\ \ldots,\ \delta_{[j]_{m,b}},\ \ldots,\
\delta_{[d-1]_{m,b}}\right].
\end{eqnarray}Note that
\[\Delta_{0,b,d}=
\text{diag} \left[ \delta_{0\mbox{ mod } b},\ \delta_{1\mbox{ mod
} b},\ \ldots,\ \delta_{j\mbox{ mod } b},\ \ldots,\
\delta_{d-1\mbox{ mod } b} \right].
\]
\begin{lemma} \label{res:permu}
Let $ \pi=( \pi(0),\ \pi(1),\ \ldots,\ \pi(b-1) )$ be a
permutation of $(0,1, \ldots,b-1)$. Let
\[ P_{\pi} = \left[ e_{\pi(0), b}\
e_{\pi(1), b}\ \ldots e_{\pi(b-1), b} \right]. \] Then, $ P_{\pi}$
is a permutation matrix and  the $(i,j)$th element of $P_{\pi}^T
E_{k,b}\Delta_{0,b,b} P_{\pi}$ is given by
\[(P_{\pi}^T E_{k,b}\Delta_{0,b,b} P_{\pi})_{i,j} = \left\{
\begin{array}{ll} \delta_t &
\mbox{if } (i,j) = (\pi^{-1}(kt  \mbox{ mod } b), \pi^{-1}(t)),  \ \ 0 \le t < b\\
0 & \mbox{otherwise.}\\ \end{array} \right. \]
\end{lemma}
The proof is easy and we omit it.

In what follows, $\chiup(A)(\lambda)$ stands for the characteristic polynomial of the matrix $A$ evaluated at $\lambda$  but for ease of notation,
we shall suppress the argument  $\lambda$ and write simply $\chiup(A)$.\begin{lemma} \label{res:2} Let
$k$ and $b$ be positive integers. Then
\begin{eqnarray} \chiup \left( A_{k,b} \right) &=& \chiup \left( E_{k,b} \Delta_{0, b, b} \right).
\end{eqnarray}where, $\delta_j = \sum_{l=0}^{b-1} a_{l} \omega^{j l}, \ 0 \le j <
b$, $\omega = cos(2\pi/b) +i sin(2\pi/b)$, $i^2=-1$.
\end{lemma}
\begin{proof}
Define the $b \times b$ permutation matrix \[ P_b = \left[
\begin{array}{cc} \underline{0} & I_{b-1} \\ 1 & \underline{0}^T
\end{array} \right] .\] Observe
that for $0 \le j < b$, the $j$-th row of $A_{k,b}$ can be written
as $a^T P_{b}^{jk}$ where $P_{b}^{jk}$ stands for $jk$-th power of
$P_b$. From direct calculation, it is easy to verify that $P_b =
UDU^*$ is a spectral decomposition of $P_b$ where
\begin{eqnarray}
D&=&\mbox{diag}(1,\omega,\ldots,\omega^{b-1}),  \\
 U &=& [ u_0\  u_1\  \cdots \  u_{b-1} ] \mbox{ with } u_j
= b^{-1/2} (1, \omega^{j}, \omega^{2j}, \ldots, \omega^{(b-1)j} ), \ 0 \le j < b.
\end{eqnarray} Note that
 $\delta_j  = a^T u_j, \ 0 \le j < b.$ From easy
computations, it now follows that
\[ U^*A_{k,b}U  = E_{k,b} \Delta_{0,b,b},  \]
so that, $\chiup \left( A_{k,b} \right) = \chiup \left( E_{k,b}
\Delta_{0,b,b} \right)$, proving the lemma.
\end{proof}
\begin{lemma} \label{res:3} Let
$k$ and $b$ be positive integers and, $x = b/gcd(k, b)$. Let for dummy variables $\gamma_{0},\ \gamma_{1},\ \gamma_{2},\ldots,
\gamma_{b-1}$,
\[\Gamma = diag \left( \gamma_{0},\ \gamma_{1},\ \gamma_{2},\ldots,
\gamma_{b-1} \right).\] Then
\begin{eqnarray} \chiup \left( E_{k,b} \times \Gamma \right)
&=& \lambda^{b-x} \chiup \left( E_{k,x} \times diag \left(
\gamma_{0 \text{ mod }b},\ \gamma_{k \text{ mod }b},\ \gamma_{2k
\text{ mod }b },\ldots, \gamma_{(x-1)k \text{ mod } b} \right)
\right).
\end{eqnarray}
\end{lemma}
\begin{proof}
Define the following matrices
\[ B_{b \times x} = \left[ e_{0,b} \ e_{k,b} \
e_{2k,b} \ \ldots \ e_{(x-1)k,b} \right] \mbox{\ and\ \ } P =
\left[ B \ B^c \right] \] where $B^c$ consists of those columns
(in any order) of $I_b$ that are not in $B$. This makes $P$ a
permutation  matrix.

Clearly, $E_{k,b} = \left[ B \  B \ \cdots \
B \right]$ which is a $b \times b$ matrix of rank $x$, and we have
 \[ \chiup \left( E_{k,b} \Gamma \right) = \chiup \left( P^T E_{k,b} \Gamma P \right). \] Note
that,
\[\begin{array}{lcl} P^T E_{k,b} \Gamma P & = & \left[
\begin{array}{llcl} I_{x} & I_{x} & \ldots & I_{x} \\
0_{(b-x) \times x} &
0_{(b-x) \times x} & \ldots & 0_{(b-x) \times x} \\
\end{array} \right] \Gamma P  \\ & & \\ & = &
\left[ \begin{array}{c} C \\ 0_{(b-x) \times b} \\ \end{array}
\right] P
\\ & & \\ & = &   \left[ \begin{array}{c} C \\ 0_{(b-x) \times b}  \\
\end{array} \right] \left[ B \  B^c \right]  = \left[ \begin{array}{cc} CB & CB^c \\ 0 & 0
\\ \end{array}
\right]
\end{array} \]
where, \[
\begin{array}{lcl} C & =
& \left[ I_{x} \  I_{x} \  \cdots \  I_{x} \right] \Gamma \\
& & \\ & = &\left[ I_{x} \  I_{x} \  \cdots \  I_{x} \right]
\times \text{diag} (\gamma_{0},\ \gamma_{1},\ \ldots,\ \gamma_{b-1}). \\
\end{array}
\]
Clearly, the characteristic polynomial of $ P^T E_{k,b} \Gamma P$
does not depend on $CB^c$, explaining  why we did not bother to
specify the order of columns in $B^c$. Thus  we have,
\[ \chiup \left( E_{k,b} \Gamma \right) = \chiup \left( P^T E_{k,b}
\Gamma P \right) = \lambda^{b-x} \chiup \left( CB \right).
\]
It now remains to show that $CB=E_{k,x} \times \text{diag} \left(
\gamma_{0 \text{ mod }b},\ \gamma_{k \text{ mod }b},\ \gamma_{2k
\text{ mod }b },\ldots, \gamma_{(x-1)k \text{ mod } b} \right)$.
Note that, the $j$-th column of $B$ is $e_{jk,b}$. So, $j$-th
column of $CB$ is actually the $(jk \mbox { mod } b )$-th column
of $C$. Hence, $(jk \mbox { mod } b )$-th column of $C$ is
$\gamma_{jk \text{ mod } b\ }e_{jk \text{ mod } x}$. So,
\[ CB = E_{k,x} \times \text{diag} \left( \gamma_{0 \text{ mod }b},\
\gamma_{k \text{ mod }b},\ \gamma_{2k \text{ mod }b },\ldots,
\gamma_{(x-1)k \text{ mod } b} \right)
\] and the Lemma is proved completely.
\end{proof}
\begin{proof}{\it of Theorem \ref{theo:formula}}. We first prove
the Theorem for $A_{k ,
n^{\prime}}$.
Since $k$ and $n^{\prime}$ are relatively prime, by Lemma \ref{res:2},
 \[
\chiup(A_{k,n^{\prime}}) =\chiup(E_{k,n^{\prime}}
\Delta_{0,n^{\prime},n^{\prime}}). \] Get the sets $S_0$, $S_1$,
$\ldots$ to form a partition of $\{0, 1, \ldots, n'-1\}$, as in
Section \ref{section:eigenvalues}.

Define the permutation $\pi$ on the set $\mathbb Z_{n^{\prime}}$
by setting $\pi(t) = s_t$, $0 \le t < n^{\prime}$.  This
permutation $\pi$ automatically yields a permutation matrix
$P_{\pi}$ as in Lemma \ref{res:permu}.

Consider the positions of
$\delta_v$ for $v \in S_j$ in the product
$P_{\pi}^TE_{k,n^{\prime}}\Delta_{0,n^{\prime},n^{\prime}}P_{\pi}$.
Let $N_{j-1}=\sum_{t=0}^{j-1}|S_t|$. We know, $S_j = \{r_jk^x
\text{ mod } n^{\prime}, x \ge 0\}$ for some integer $r_j$. Thus,
\[ \pi^{-1}\left( r_jk^{t-1} \mbox{ mod } n^{\prime} \right) = N_{j-1}+t,\ \ 1 \le t \le
n_j \] so that, position of $\delta_v$  for $v=r_jk^{t-1} \mbox{
mod } n^{\prime}  $, $1 \le t \le n_j$ in $P_{\pi}^T
E_{k,n^{\prime}}\Delta_{0,n^{\prime}}P_{\pi}$ is given by \[
\left( \pi^{-1}(r_jk^y \mbox{ mod } n^{\prime} ),
\pi^{-1}(r_jk^{y-1} \mbox{ mod } n^{\prime} ) \right) = \left\{
\begin{array}{ll} \left( N_{j-1}+t+1,\ N_{j-1}+t \right) &
\mbox{if, } 1 \le t < n_j \\ \left( N_{j-1}+1,\ N_{j-1}+n_j
\right) & \mbox{if, } t = n_j \\ \end{array} \right. \]
Hence, \[
P_{\pi}^TE_{k,n^{\prime}}\Delta_{0,n^{\prime},n^{\prime}}P_{\pi} =
\text{diag} \left( L_0,\ L_1,\ \ldots \right) \] where, for $j \ge
0$, if $n_j=1$ then $L_j = \left[ \delta_{r_j} \right]$ is a $1
\times 1$ matrix, and if $n_j > 1$, then,
\[ L_j = \left[ \begin{array}{cccccc} 0 & 0 & 0 & \ldots & 0 &
\delta_{r_jk^{n_j-1} \text{ mod } n^{\prime} }
\\ \delta_{r_j \text{ mod } n^{\prime}} & 0 & 0 & \ldots & 0 & 0 \\ 0 & \delta_{r_jk \text{ mod } n^{\prime}} & 0 &
\ldots & 0 & 0\\ & & & \vdots & \\
0 & 0 & 0 & \ldots & \delta_{r_jk^{n_j-2} \text{ mod } n^{\prime}} & 0. \\
\end{array}\right] \] Clearly, $\chi(L_j) = \lambda^{n_j} - \Pi_j$. Now the result
follows from the identity \[ \chiup \left( E_{k,n^{\prime}}
\Delta_{0,n^{\prime},n^{\prime}} \right) = \prod_{j \ge 0}
\chiup(L_j) = \prod_{j \ge0} ( \lambda^{n_j} - \Pi_j ).
\]
Now let us prove the results for the general case.
Recall that $n=n^{\prime} \times \Pi_{q=1}^{c} p_q^{\beta_q}$.
Then, again using Lemma \ref{res:2},
\[ \chiup (A_{k,n}) = \chiup(E_{k,n} \Delta_{0,n,n}). \]
Recalling Equation \ref{eq:modrelation}, Lemma \ref{res:2}
and using Lemma \ref{res:3} repeatedly,
\[ \begin{array}{lcl} \chiup (A_{k,n}) & = & \chiup (E_{k,n} \Delta_{0,n,n})\\ &
=& \lambda^{n-n^{\prime}} \chiup(E_{k,n^{\prime}}
\Delta_{m,n,n^{\prime}})\\ & = & \lambda^{n-n^{\prime}}
\chiup(E_{k,n^{\prime}} \Delta_{m+j,n,n^{\prime}}) \ \hfill \ [\text{ for all }
j \ge 0]\\ & = & \lambda^{n-n^{\prime}} \chiup
\left(E_{k,n^{\prime}} \times \text{diag}\left( \delta_{[0]_{0,n}},\
\delta_{[y]_{0,n}},\ \delta_{[2y]_{0,n}},\ldots,
\delta_{[(n^{\prime}-1)y]_{0,n}} \right) \right) \ \hfill \  [ \text{where }
 y = n/n^{\prime}].
\end{array}
\] Replacing
$\Delta_{0,n^{\prime},n^{\prime}}$ by $\text{diag} \left(
\delta_{[0]_{0,n}},\ \delta_{[y]_{0,n}},\
\delta_{[2y]_{0,n}},\ldots, \delta_{[(n^{\prime}-1)y]_{0,n}}
\right)$, we can mimic the rest of the proof given for
$A_{k,n^{\prime}}$, to complete the proof in the general case.
\end{proof}

\footnotesize
\noindent Address for correspondence:\\
\noindent Arup Bose\\
Stat-Math Unit\\
Indian Statistical Institute\\
203 B. T. Road\\
Kolkata 700108\\
INDIA\\
email: bosearu@gmail.com

\end{document}